\documentclass{amsart}

\usepackage{amsmath}
\usepackage{amssymb}
\usepackage{amsfonts}
\usepackage{amsthm}
\usepackage{verbatim}

\newcommand{\nc}{\newcommand}
\nc{\nt}{\newtheorem}
\nc{\dmo}{\DeclareMathOperator}
\nc{\enm}{\ensuremath}

\newtheorem{thm}{Theorem}

\newtheorem{prop}{Proposition}
\newtheorem{lemma}{Lemma}
\nt{defn}{Definition}
\newtheorem{cor}{Corollary}
\nt{assumption}{Assumption}

\dmo{\Ind}{Ind}
\dmo{\cInd}{c-Ind}
\dmo{\Adj}{Ad}

\dmo{\SO}{SO}
\dmo{\Lie}{Lie}

\dmo{\reg}{reg}
\dmo{\sing}{sing}
\dmo{\supp}{supp}
\dmo{\tr}{tr}
\dmo{\Sym}{Sym}
\dmo{\Hom}{Hom}
\dmo{\Tor}{Tor}
\dmo{\Out}{Out}
\dmo{\Ht}{ht}
\dmo{\End}{End}
\dmo{\Mat}{Mat}

\dmo{\SL}{SL}
\dmo{\sgn}{sgn}
\dmo{\GL}{GL}
\DeclareMathOperator*{\Res}{Res}
\dmo{\Mod}{mod}
\dmo{\geo}{geo}
\dmo{\re}{Re}
\dmo{\Spec}{Spec}
\dmo{\Fr}{Fr}
\dmo{\vol}{vol}
\dmo{\Sets}{Sets}
\dmo{\im}{im}
\dmo{\diag}{diag}
\dmo{\Ker}{Ker}
\dmo{\val}{val}
\dmo{\ord}{ord}
\dmo{\Stab}{Stab}
\dmo{\Ad}{Ad}
\dmo{\rank}{rank}
\dmo{\Symp}{Sp}

\nc{\Aff}{\mathbb{A}}
\nc{\eps}{\varepsilon}
\nc{\bks}{\enm{{\backslash}}}
\nc{\isom}{\enm{{\overset{~}{\rightarrow}}}}
\nc{\Z}{\enm{{\mathbb Z}}}
\nc{\Zp}{\enm{{\mathbb Z_p}}}

\nc{\KL}{\mathcal L}

\nc{\bG}{\enm{{\mathbf G}}}
\nc{\bT}{\enm{{\mathbf T}}}
\nc{\bH}{\enm{{\mathbf H}}}

\nc{\Gm}{\enm{{\mathbb G_m}}}
\nc{\F}{\enm{{\mathbb F}}}
\nc{\Fp}{\enm{{\mathbb F}_p}}
\nc{\Fq}{\enm{{\mathbb F}_q}}
\nc{\Q}{\enm{{\mathbb Q}}}
\nc{\Qp}{\enm{{\mathbb Q_p}}}
\nc{\R}{\enm{{\mathbb R}}}
\nc{\N}{\enm{{\mathbb N}}}
\nc{\C}{\enm{{\mathbb C}}}
\nc{\CC}{\enm{{\mathcal C}}}
\nc{\half}{\enm{{\frac{1}{2}}}}
\nc{\BB}{\enm{{\mathcal B}}}
\nc{\flip}{\tilde{\eps}}

\nc{\ii}{\enm{{\mathcal I}}}
\nc{\jj}{\enm{{\mathcal J}}}
\nc{\OO}{\enm{{\mathcal O}}}
\nc{\f}{\enm{{\mathcal F}}}

\nc{\GGl}{\enm{{\mathfrak gl}}}
\nc{\GG}{\enm{{\mathfrak g}}}
\nc{\gd}{\enm{{\hat{\mathfrak g}}}}
\nc{\gm}{\enm{{\gamma}}}
\nc{\hh}{\enm{{\mathfrak h}}}

\nc{\II}{\enm{{\mathfrak a}}}
\nc{\LL}{\enm{{\mathfrak l}}}
\nc{\mm}{\enm{{\mathfrak m}}}
\nc{\pp}{\enm{{\mathfrak p}}}
\nc{\TT}{\enm{{\mathfrak t}}}
\nc{\Nc}{\enm{{\mathcal N}}}
\nc{\Cc}{\enm{{\mathcal C}}}

\nc{\Rsing}{R_{\sing}}

\nc{\Gd}{\enm{{\hat{G}}}}
\nc{\Hd}{\enm{{\hat{H}}}}

\nc{\vt}{\enm{\vartheta}}
\nc{\lra}{\enm{\longrightarrow}}
\nc{\ra}{\enm{\rightarrow}}
\nc{\lip}{\enm{\langle}}
\nc{\rip}{\enm{\rangle}}

\nc{\bsk}{\bigskip}
\nc{\ol}{\overline}

\author{Steven Spallone }

\begin{document}

\title{Residues of Intertwining Operators for Classical Groups}
\maketitle
\begin{center} {\it with an Appendix ``$L$-Functions and Poles of Intertwining Operators'' by Freydoon Shahidi} \end{center}

\begin{abstract}

Let $\tilde{G}$ be a symplectic or even orthogonal group over a p-adic field $F$, and $M$ the Levi factor of a maximal parabolic subgroup of $\tilde{G}$.  Suppose that $M$ has the shape of three blocks of the same size.  Let $\pi$ be a supercuspidal representation of $M$.  In this paper we give a simple explicit expression for the residue of the standard intertwining operator for the parabolic induction of $\pi$ from $M$ to $G$. 

\end{abstract}

\section{\bf Introduction}

This paper continues a study of the reducibility of a representation of a classical group $\tilde{G}$ induced from a supercuspidal representation of a Levi factor $M$ of a maximal parabolic subgroup $P=MN$ of $\tilde{G}$. 

The problem classically reduces to the evaluation of the residue of an intertwining operator, an integral over the unipotent radical $N$.  One studies this integral by decomposing $N$ into its orbits under $M$.  It is of great interest to study the poles of this operator, as they determine certain $L$-functions attached to these representations (see \cite{Sh-A}).

Shahidi studied this question in \cite{S92}, for the case of Siegel parabolics.  This is the case for which $M$ has the shape of two blocks and is isomorphic to $\GL_n(F)$.  The group $N$ is isomorphic to a subgroup of the additive group $M_n(F)$, and the action of $M$ on $N$ is twisted conjugacy, as studied in \cite{KS}.  The word ``twist'' refers to an automorphism $\eps$ of $\GL_n(F)$ conjugate to inverse-transpose.  He reduces the integral to a sum of twisted orbital integrals, and the question of reducibility becomes that of twisted endoscopic transfer, in the sense of \cite{KS}.

Following \cite{S95}, Goldberg and Shahidi pursued the problem in \cite{GS98} and \cite{GS01} for general maximal parabolics.  They considered Levi subgroups $M$ of three blocks, being isomorphic to the product $G=\GL_n(F)$ with a smaller classical group $H$.  In this paper we focus on the case when $H$ and $G$ have the same size.  The  representation of $M$ is given by the tensor product $\pi_G \otimes \pi_H$ of supercuspidal representations of $G$ and $H$.  
Especially interesting is the case when $\pi_G$ is self-dual; otherwise the induced representation is automatically irreducible.
Write $\omega$ for the central character of $\pi_G$; we must have $\omega^2=1$.

The unipotent radical $N$ is no longer abelian, and the geometry of action of $M$ on $N$ becomes much richer.
The residue is reduced to the sum of two terms, written symbolically as
\[ R(f_G,f_H)=c \cdot R_G(f_G,f_H)+R_{\sing}(f_G,f_H),\]
with $c=\frac{1}{2n \log q}$.

Here $f_H$ is a matrix coefficient for $\pi_H$, and $f_G$ is a compactly supported function on $G(F)$ for which
\begin{equation} \label{psi}
 \psi(g)= \int_{Z(G)} \omega(z)^{-1}f_G(zg) dz
\end{equation} 
is a matrix coefficient of $\pi_G$. 

The intertwining operator will be holomorphic if the quantity $R(f_G,f_H)=0$ for all 
choices of $f_G$ and $f_H$.

The term $R_G(f_G,f_H)$ is a sum of integrals of the form
\begin{equation} \label{pairing}
 \int_T I(\gm,f_H) I_\eps(\delta,f_G) d\gm. 
\end{equation} 

Here $T$ is an elliptic Cartan subgroup of $H$, and $I$ denotes a normalized orbital integral.  The element $\delta$ corresponds to $\gm$ under the norm correspondence of [KS], and $I_\eps$ denotes a normalized twisted orbital integral. 
The fact that the norms introduced in \cite{GS98},\cite{GS01} are the same as those in \cite{KS} was first observed in \cite{S95}.
This expression suggests ``Schur orthogonality'' methods, but for two {\it different} groups.

The term $\Rsing(f_G,f_H)$ is analytically more complex; it may be written as a sum, over the maximal tori $T$ of $H$, of limits of residues of integrals of the form
\[ \lim_{C_T} \Res_{s=0} \int_{T_r-C_T} \psi(s,\gm) |D_\eps(\gm)| d\gm, \]
where $\psi(s,\gm)$ is a function depending on $s, \gm, f_G,f_H$, and two compact subsets of $M_n(F)$.  We will
specify $\psi(s,\gm)$ more precisely in the next section.
Here $T$ is a Cartan subgroup of $H$, $T_r$ is its subset of regular elements, and the limit runs over compact subsets of $T_r$.  The function $D_\eps(\gm)$ is a twisted version of the usual Weyl discriminant.

These two terms arose in the following way.  The original problem reduces to computing the residue of an expression of the form
\begin{equation} \label{Isum}
 I(s,f_G,f_H)=\sum_T |W(T)|^{-1} \int_T \psi(s,\gm)|D_\eps(\gm)|d \gm. 
\end{equation}

Here $T$ runs over conjugacy classes of maximal tori in $H$, and $W(T)$ denotes the Weyl group of $T(F)$ in $H(F)$.

We have $R(f_G,f_H)=\Res_{s=0} I(s,f_G,f_H)$.

The analysis of the function $\psi(s,\gm)$ goes smoothly when $\gm$ is constrained to compact subsets $C_T$ of regular elements.  This led Goldberg-Shahidi to study $I(s,f_G,f_H)$ as a ``principal value integral''; then $R_G$ captures the regular part of the residue, and $\Rsing$ captures the contribution to the residue near singular points of $T$.
If, in the expression for $\Rsing$, the limit and the residue are switched, the result is $0$.  However we do not expect the quantity $\Rsing$ itself to always vanish; therefore the convergence must be conditional.  (See \cite{Sh-A}.)

The details of \cite{GS98} are reviewed in Section 2.

In this paper we take a different approach to the residue.  Rather than taking the ``principal value'' approach, we analyze the more primal function $I(s,f_G,f_H)$ directly.  

The crux of the divergence of $I(s,f_G,f_H)$ lies in the integral (\ref{psi}) over $z \in Z(G)$.  This integral, and thus $I(s,f_G,f_H)$, breaks up as an infinite sum according to the norm $q^k$ of $z$.  The term for a fixed $k$ converges, and it makes sense to treat $I(s,f_G,f_H)$ as a power series in $q^{-s}$.

This inspires us to switch the sum past a few integrals; for this purpose we need some estimates on the integrand. 
These estimates are of the type designed to prove convergence for the local trace formula (see \cite{K}), but we require twisted analogues. 

In Section 3, we prove that the twisted centralizer and twisted normalizer of $S(\gm)^{-1}$ are both equal to $T$, the latter up to finite index, and prove
\begin{prop}  The map \[ \beta: G/T \times T_r \ra G_{\eps rs} \] 
given by $\beta(g,\gm)= g (S(\gm)^{-1}) g^{\vdash}$ is a finite map. \end{prop}

Here $G_{\eps rs}$ is the set of $\eps$-regular $\eps$-semisimple elements in $G$, in the sense of \cite{KS}.

Section 4 is mostly a review of some standard estimates from Harish-Chandra's theory of orbital integrals.
We also sketch a proof suggested in \cite{K} of the local integrability of $|D_\eps(\gm)|^{-\eps}$.

Throughout all of this an open compact subset $L' \subset M_n(F)$ has been fixed.  This ``lattice'' originates from the local constancy of a function in the induced space whose irreducibility we are studying.  When $L'$ is $\OO$-invariant, the quantity $w_k(g,h)$ is given by 
\[ w_k(g,h)=\vol_T(T \cap \varpi^{-k}g^{-1}L'h^{-1}). \]

Sections 5 and 6 anticipate the importance of this ``weight factor'', which comes out of the integrals of \cite{GS98}, and prove that if $f_G(g\delta^{-1} g^{\vdash}) \neq 0$ and $f_H(h^{-1} \gm h) \neq 0$, then is a constant 
$c_1>0$, and a locally integrable function $\Phi(\gm)$ on $T_r$ so that 

\begin{enumerate}
\item If $2k+ c_1 + \Phi(\gm) <0,$ then $ w_k(g,h)=0$. 
\item If $2k+ c_1 + \Phi(\gm) \geq 0$ then
\[  w_k(g,h) \leq c_L (2k+c_1+\Phi(\gm))^r. \]
Here $r$ is the split rank of $T$.
\end{enumerate}

We apply these estimates in Section 7 to switch the sum in $k$ outside the integrals, which leads to considerable simplification.  Here is our main theorem, an expression for the residue as a ``pairing'' between an orbital integral for $H$ and a (twisted) orbital integral for $G$.

\begin{thm}  \label{Thm1}
The residue $R(f_G,f_H)$ is equal to
\[\Res_{s=0}  \sum_T |W(T)|^{-1} \sum_{k=0}^{\infty} q^{-2nks} \int_T |D_\eps(\gm)| \int_{G/T} \int_{T \backslash H^+} f_G(g \delta^{-1}g^{\vdash}) f_H(h^{-1}\gm h)  W_k(g,h) dh dg d\gm. \] \end{thm}
  
Here the quantity $W_k(g,h)$ is a sum
\[ W_k(g,h)= \sum_{\alpha} \omega(\alpha)^{-1} w_k(g x_\alpha^{-1},h), \] 
where $\alpha$ runs over the square classes and the $x_\alpha \in G$ are diagonal matrices with $x_\alpha \eps(x_\alpha)^{-1}=\alpha \cdot I$.

Please note that if $W_k(g,h)$ were constant, the integrals in Theorem \ref{Thm1} would factor simply into the product of two orbital integrals.  So we view $W_k(g,h)$ as a  ``weight factor'', akin to those appearing in the weighted integrals of the local trace formula \cite{Art}, but curiously mixing orbital integrals on both $G(F)$ and $H(F)$.

At present our work covers the symplectic and quasi-split even orthogonal cases, since for these the norm correspondence is generically an injection; indeed if $\gm \in T$ with $\gm-I$ invertible, then we may take $\delta=S(\gm)=wJ^{-1}(\gm-I)$ as the preimage.  More generally, the fibers will be finite, according to Lemma 3.11 of \cite{GS98}.  Such a finite sum should not affect the analysis, we expect our results to extend to all quasi-split classical groups.

This describes the first part of this paper.

To demonstrate that it is feasible to calculate with Theorem \ref{Thm1}, we perform a sample computation in the second part of this paper.
We study the case in which $\tilde{G}=\SO(6)$, $G=\GL(2)$, and $H$ is split $\SO(2)$.  This does not give a {\it maximal} parabolic, but the case is simple enough so that many of the ingredients can be made explicit.

As our test case, we take the representation on $H$ to be trivial, and the representation on $G$ to be one of those given in \cite{Ku}, and coming from a ramified quadratic extension $E$ of $F$.  These are representations which are compactly induced from characters on a compact mod center subgroup of the form $E^{\times}\KL$, where $\KL$ is an appropriate compact open subgroup.

In Section 8 we compute $f_G$ and $W_k(g,h)$ in this situation.  We have
\[ R(f_G,f_H)= 2 \Res_{s=0} \sum_{k=0}^{\infty} q^{-4ks} \int_{T_r} |D_\eps(\gm)|\int_{G/T} f_G(gS(\gm)^{-1}g^{\vdash}) W_k(g) \frac{dg}{dt} d\gm. \]
Here $W_k(g)=W_k(g,1)$.

We may disregard most values of $\gm \in T_r$, for the following reason.

Write  $\gm \in T$ as $\left( \begin{array}{cc} 
\alpha & 0  \\ 
0 &  \alpha^{-1} \\
 \end{array} \right)$, with $\alpha \in F^{\times}$, and $\alpha \neq \pm 1$.
Then in Section 9 we show that
if $S(\gm)^{-1}$ is only $\eps$-conjugate to a matrix in the support of $f_G$, then we must have $\alpha \in \OO^{\times}$ and in fact
$\alpha = \pm 1$ mod $\pp$.
Moreover, if the residue characteristic of $F$ is odd we must also have $\alpha=-1$ mod $\pp$.

If the residue characteristic is odd and $\alpha=-1$ mod $\pp$, the integral may not vanish, but
the weight factor becomes constant, and one may factor out an ordinary twisted orbital integral from the computation, and the analysis becomes trivial.
In Section 10, we study the final case, in which the residue characteristic is even and $\alpha=1$ mod $\pp$.  Here we have a nonvanishing result, in which the weight factor and orbital integral interact.  We call the reader's attention to the fact that the analysis of 
$I(s,f_G,f_H)$ in this case is concentrated near the singular points of $T$, in the sense that it remains the same if a compact subset of $T_r$ is removed.  

These computations serve as a model for the study of the functions $I(s,f_G,f_H)$, showing how the analysis of the weighted integral should resolve itself into ``regular'' and ``singular'' terms.
This concludes the discussion of the second part of the paper.

This project was carried out while the author was a Research Assistant Professor at Purdue University, and he is grateful for the support of the department.  It was especially invaluable to work with the constant guidance and encouragement of Freydoon Shahidi, who suggested the problem.  The author would also like to thank David Goldberg, Jiu-Kang Yu, and Robert Kottwitz for their interest and valuable discussions on the project.

\section{\bf Review of Goldberg-Shahidi, Notation} \label{GS}

The purpose of this section is to review the origins of the ingredients of the function $\psi(s,\gm)$ appearing in the expression (\ref{Isum}) for the function $I(s,f_G,f_H)$.  Details may be found in \cite{GS98}.

Let $F$ be a $p$-adic field of characteristic zero.  Write $\OO$ for its ring of integers and $\varpi$ for a uniformizer.  Let $q$ be the order of the residue field.

In what follows we will use boldface, e.g., $\bG$ to denote an algebraic group defined over $F$, and $G$ to denote its set of $F$-points $G(F)$.

The theory for symplectic groups and quasi-split even orthogonal groups is similar, and much can be done in parallel.

Let $m$ be a positive integer, and $n=2m$.
For a positive integer $i$, let $w_i$ be the permutation matrix of size $i$ with $1$s down the antidiagonal.
Let $\Lambda$ be a $2 \times 2$ invertible symmetric matrix.  For a positive even integer $\ell$, consider the matrix
\begin{equation} \label{orthogonal}
J_\ell=J^{\Lambda}_\ell=\left( \begin{array}{ccc} 
 & & w_i \\ 
 & \Lambda & \\
 w_i & &  \\
 \end{array} \right), 
\end{equation}
where $i$ is chosen so that $2i+2=\ell$.
 
Also for a positive even integer $i$ write $u_i$ for the antidiagonal matrix
\begin{equation} \label{symplectic}
u_i=\left( \begin{array}{cccccc} 
 &&&& &  \cdot \\ 
 &&&& \cdot & \\
 &&& \cdot& &  \\
  &&-1&& &  \\ 
 &1&&& & \\
-1 &&&& &  \\
 \end{array} \right), 
\end{equation}
of size $i$.

For orthogonal groups, having fixed $\Lambda$, write $J_{3n}$ for the matrix given by (\ref{orthogonal}).
Then we define $\tilde{G}=\SO(J_{3n})$ to be the special orthogonal group defined with respect to $J_{3n}$;
thus $\tilde{G}$ is the connected component of $\{ g \in \GL(3n)| gJ_{3n}{}^tg=J_{3n} \}$.

For symplectic groups, write $J_{3n}$ for the matrix $u_{3n}$ given by (\ref{symplectic}).
Then $\tilde{G}=\Symp_{3n}(F)$ is the usual group of symplectic matrices over $F$; it is connected.

For $g \in \GL_n(F)$, write $g^\vdash=w_n \cdot {}^tgw_n^{-1}$ in the orthogonal case, and $g^\vdash=u_n {}^tgu_n$ in the symplectic case.

%This is simply the ``transposition'' of $g$ along the second diagonal.  

Let $\eps(g)=(g^{-1})^\vdash$; this is an involution of $G$.

In the orthogonal case, write $H^+$ for the group O$(J_n)$, and $H$ for the connected component $\SO(J_n)$. 

In the symplectic case, write $H^+=H=\Symp_n(F)$.
Write $M$ for the subgroup of matrices of the form
\[ \left( \begin{array}{ccc} 
g & &  \\ 
 & h & \\
  & & \eps(g) \\
 \end{array} \right), \]
with $g \in \GL_n(F)$ and $h \in H$.
Write $P$ for the parabolic subgroup generated by $M$ and the Borel of upper triangular matrices in $\tilde{G}$.
Then $P=MN$, where $N$ is the subgroup of matrices of the form

\[ n(X,Y) = \left( \begin{array}{ccc} 
 I&X & Y \\ 
 & I & X'\\
  & & I \\
 \end{array} \right) \]
in $\tilde{G}$.  Here, $X,X'$ and $Y$ are $n \times n$ blocks.  The condition that $n \in \tilde{G}$ gives the equation
\begin{equation} \label{orthogonalstar}
 X'=-J_n {}^tX w_n \text{ and } Y+Y^{\vdash}=XX'
 \end{equation}
in the orthogonal case, and 

\begin{equation} \label{symplecticstar}
X'=u_n {}^tX u_n  \text{ and } Y+Y^{\vdash}=XX'
\end{equation}
in the symplectic case.
 
The group $M^+=G \times H^+$ acts on $N$ via the adjoint action. 

Let $\pi_G$ and $\pi_H$ be irreducible unitarizable supercuspidal representations of $G$ and $H$ respectively, with $\pi_G$ self-dual.  The central character $\omega$ of $\pi_G$ satisfies $\omega^2=1$.
Their tensor product is an irreducible unitarizable supercuspidal representation of $M$.  We wish to study its parabolic induction $\pi=I(\pi_G \otimes \pi_H)$ to $\tilde{G}$.  

Consider the family of induced representations $I(s, \pi_G \otimes \pi_H)=\Ind_P^{\tilde{G}} (\pi_G \otimes |\det|^s \otimes \pi_H \otimes 1_N)$.  Here $s \in \C$ with $\re(s)>0$. 

Write $w_0$ for the permutation matrix given by
\[ \left( \begin{array}{ccc} 
 & & I \\ 
 & I & \\
 I & &  \\
 \end{array} \right). \]
One has an intertwining operator
$A=A(s, \pi_G \otimes \pi_H,w_0)$ on $I(s, \pi_G \otimes \pi_H)$ given by the formula
\[  (A(s, \pi_G \otimes \pi_H,w_0)f)(g)= \int_N f(w_0^{-1}ng)dn. \]
 
It is of interest to determine the pole of $A$ at $s=0$.  In fact if $\pi_H$ is generic then the poles of this operator are the same as the poles of the product of $L$-functions $L(s, \pi_G \times \pi_H)L(2s,\pi_G,
\wedge^2 \rho_n)$, in the notation of \cite{GS98}.
To find these poles, one in principle must test all functions $f \in I(s, \pi_G \otimes \pi_H)$.  

By a lemma of Rallis \cite{S95}, it is enough to compute the poles that arise when $A$ is applied to functions $h \in V(s, \tau' \otimes \tau)_0$ and evaluated at the identity.  These functions $h$ are determined by their restriction to $\overline{N}$, the transpose of $N$, modulo $P$.  We may assume that there is a vector $v' \otimes v \in \pi_G \otimes \pi_H$ and compact subsets $L,L' \subset M_n(F)$, with $L'$ open, so that

\[ f\left(\left( \begin{array}{ccc} 
 I&0 & 0 \\ 
 X' \eps(Y)& I & 0\\
 Y^{-1} & Y^{-1}X & I \\
 \end{array} \right)\right)= \xi_L(Y^{-1}) \xi_{L'}(Y^{-1}X)(v' \otimes v), \]
where we write $\xi_S$ for a characteristic function of a set $S$.  Please see Remark 9 of \cite{Sh-A} for a complete discussion of $L$ and $L'$. 
  
We now argue that we may assume $0 \in L'$.
Let $L$ and $L'$ be compact subsets in $M_n(F)$.  Write $h_{L,L'}$ for the function in $C_c^{\infty}(\overline{N})$ satisfying 

\[ h_{L,L'} \left( \left( \begin{array}{ccc}
I & 0 & 0 \\
C & I & 0 \\
A & B & I \\
\end{array} \right) \right)=\xi_L(A)\xi_{L'}(B). \]

$C_c^{\infty}(\overline{N})$ is spanned by such functions, but we argue that it is also spanned by such functions where $0 \in L'$.  Suppose that $0$ is not in a given $L'$.  Pick an open compact subset $M'$ in $M_n(F)$ containing $0$ but disjoint from $L'$.  Then 
\[ h_{L,L'}= h_{L,M' \cup L'} - h_{L,M'}. \]

The functions $h$ in \cite{GS98} are obtained by tensoring this space with $V' \otimes V$; the resulting space is then isomorphic to $V(s, \pi_G \otimes \pi_H)_0$.  Thus we will henceforth assume that $L'$ contains $0$.  

\bigskip

Pick vectors $\tilde{v}'$ and $\tilde{v}$ in the dual space of $\pi_G \otimes \pi_H$.  Write $\psi$ and $f_H$ for matrix coefficients of $\pi_G$ and $\pi_H$ given by the pairs $(v',\tilde{v}')$ and $(v,\tilde{v})$.  
The function $\psi$ has central character $\omega$, and is not compactly supported.
However we may choose a smooth compactly supported function $f_G$ from which we may recover $\psi$ by
\[ \psi(g)= \int_{Z(G)} \omega(z)^{-1}f_G(zg) dz . \]

Here $Z(G)$ denotes the center of $G$.

Write $L^\vdash=\{ \ell^\vdash| \ell \in L \}$.  Then the pairing $ \lip \tilde{v}' \otimes \tilde{v},A(s, \pi_G \otimes \pi_H,w_0)f(I) \rip$ is given by

\[ I(s,f_G,f_H)= \int_{n(X,Y)} \int_{F^{\times}}\omega(z)^{-1} f_G(zY) f_H(I-X'Y^{-1}X)|\det Y|^s \xi_{L^\vdash}(Y) \xi_{L'}(X)dz d^*(X,Y), \]
where $d^*(X,Y)$ denotes an $M^+$-invariant measure on $N$.

At this point we may write $R(f_G,f_H)=\Res_{s=0} I(s,f_G,f_H)$.

One handles the integral by breaking up $N$ into orbits under $M^+$.  For $(g,h) \in M^+$, we have $\Ad(g,h)n(X,Y)=n(gXh^{-1},gYg^{\vdash})$. 

For almost all $n(X,Y)$, the matrix $X$ is invertible, so we may pick representatives of orbits of $N$, under the action of $M^+$, of the form $(I,Y)$.  Considering the action of $(g,g)$, we may allow such $Y$ to run over representatives for $\eps$-regular, $\eps$-semisimple $\eps$-conjugacy classes in $\GL_n(F)$.  

This approach breaks the problem into two parts: First, to parametrize all the orbits, and second, to determine contribution from the orbit of a given $n(I,Y)$.

\bigskip
The solution to the first part of the problem involves the norm correspondence from twisted endoscopy. 

One studies the map $n(X,Y) \mapsto I-X'Y^{-1}X$, to relate the arguments of $f_H$ and $f_G$.

Write $\Nc$ for the set of all $\eps$-conjugacy classes of elements $Y \in \GL_n(F)$ for which there exist $X \in \GL_n(F)$ so that Equation (\ref{orthogonalstar}) or (\ref{symplecticstar}) is satisfied.  This is closed under inversion. Write $\Cc$ for the set of conjugacy classes in $H$.
 
 We define the norm correspondence $N_\eps: \Nc \rightarrow \Cc$ by saying that the classes $\{ \delta \} \in \Nc$ and $\{ \gm \} \in \Cc$ correspond if there is an $F$-rational solution $(X,Y)$ of (\ref{orthogonalstar}) or (\ref{symplecticstar}) so that $I-X'Y^{-1}X \in \{ \gm \}$ and $Y^{-1} \in \{ \delta \}$.
Then $N_\eps$ is surjective and has finite fibers.  Moreover, if $(I,Y)$ satisfies (\ref{orthogonalstar}) or (\ref{symplecticstar}), then $N_\eps(\{ Y^{-1} \})=\{ -\eps(Y^{-1})Y^{-1} \} \in \Cc$.
 
 It is easy to see that if $\gm-I$ is invertible then there is a unique preimage $S(\gm)=wJ_n^{-1}(\gm-I)$ of $N_\eps$.
The set of such $\gm$ has full measure in $T$, and so we assume this is the case when integrating.

Here are some twisted analogues of familiar definitions.

\begin{defn} Given $\delta \in G$, write $G_{\eps,\delta}$ for the twisted centralizer of $\delta \in G$.  That is,
\[ G_{\eps,\delta}= \{g \in G| g \delta g^\vdash=\delta \}. \]
Write $G_{\eps rs}$ for the set of $\eps$-regular, $\eps$-semisimple elements, in the sense of \cite{KS}.
For $\delta \in G_{\eps rs} $ let 
\[ D_\eps(\delta)= \det(\Ad(\delta) \circ \eps -1; \Lie(G)/\Lie(G_{\eps,\delta})). \]
\end{defn} 

We will often write $D_\eps(\gm)$ for $D_\eps(S(\gm))$.  We write $T_r$ for the set of regular elements of a torus $T$.

Replacing the orbits of $Y$ with the orbits of $\gm \in H$ leads to the following change of variables for an integral over $N$ of some function $\varphi$:
\[ \int_{(X,Y)} \varphi(n(X,Y))d^*(X,Y)=\sum_T |W(T)|^{-1} \int_{\gm \in T_r}|D_\eps(S(\gm))| \left(  \int_{[n(I,S(\gm)^{-1})]} \varphi \right) d\gm. \]

  Here $T$ runs over $H$-conjugacy classes of maximal tori in $H$, and $[n(I,S(\gm)^{-1})]$ is the orbit of $n(I,S(\gm)^{-1})$ under $M^+$, whose measure will be discussed below.

\bigskip
For the second part, to understand the measure of the orbit of  $n(I,S(\gm)^{-1})$ under $M^+$,  we consider the map $M^+ \rightarrow N$ given by $(g,h) \mapsto \Ad((g,h))(I,Y)=(gh^{-1},gYg^\vdash)$.  The fibre of this over $(I,Y)$ is isomorphic to the twisted centralizer $G_{\eps,Y}$ embedded diagonally into $M^+$.  

Then the contribution from the orbit of $n(I,S(\gm)^{-1})$ to $I(s,f_G,f_H)$ is given by

\[ \psi(s,\gm)= \sum_{\alpha \in A} \omega(\alpha)^{-1} \cdot \int_{G/G_{\eps,S(\gm)^{-1}}} \int_{H^+_\gm \backslash H^+} \int_T f_G(\alpha \cdot gS(\gm)^{-1}g^{\vdash}) f_H( h^{-1} \gm h)|\det(gS(\gm)^{-1}g^{\vdash})|^s  \cdot \]
\[ \int_{Z(G)} \xi_{L^\vdash}(z^{-2}g S(\gm)^{-1}g^{\vdash})\xi_{L'}(z^{-1}gh_0h)|\det z|^{-2s}dz dh_0dhdg.\]

Here $A$ is a set of representatives for $F^{\times}/F^{\times 2}$.
 
With this notation, we have
\[ I(s,f_G,f_H)=\sum_T |W(T)|^{-1} \int_{\gm \in T_r}|D_\eps(\gm)| \psi(s,\gm) d \gm. \]

In this paper we study its residue at $s=0$.
\section{\bf Twisted Centralizers and Normalizers}

This section focuses on the even orthogonal case.  The symplectic case is similar and we omit it.  We will be using methods of algebraic geometry and all groups are considered with points in the algebraic closure $\ol{F}$ of $F$.  
Let $J=J_n$ and $w=w_n$.  

Recall that $\bG=\GL(n)$ and $\bH^+$ is the set of matrices $\{ h \in \GL(n)| h J {}^th=J \}$.
Let $\bT$ be a maximal torus in $\bH=(\bH^+)^\circ$, and write $\bT_r$ for its regular elements.
For $g \in \bG$ write $\nu(g)=\eps(g)g$.
For $\gm \in \bT$ write $S(\gm)=wJ^{-1}(\gm-I)$.  One checks that $\nu(S(\gm))=-\gm$.
  
\begin{prop} \label{tcent} Let $\bT$ be a maximal torus in $\bH$, and $\gm \in \bT_r$ a regular element.   The twisted centralizer $\bG_{\eps,S(\gm)^{-1}}=\{ g \in \bG| g S(\gm)^{-1} g^{\vdash}=S(\gm)^{-1} \}$ is equal to $\bT$.
\end{prop}
\begin{proof}

It is straightforward to check that $\bT \subseteq \bG_{\eps,S(\gm)^{-1}}$. If $g \in \bG_{\eps,S(\gm)^{-1}}$ then $\eps(g) S(\gm) g^{-1}=S(\gm)$.  Applying $\nu$ to this equation gives $g(- \gm) g^{-1}=-\gm$, thus $g$  commutes with $\gm$. Then the equation
$\eps(g)wJ^{-1}(\gm-I)g^{-1}=wJ^{-1}(\gm-I)$ gives 
$\eps(g)wJ^{-1}g^{-1}(\gm-I)=\gm-I$.  Therefore $\eps(g)wJ^{-1}g^{-1}=wJ^{-1},$ which implies that $g \in \bH^+$.
Since $Z_{\bH^+}(\gm)=T$, we conclude that $g \in \bT$.
\end{proof}

Let
\[ N_{\bG}^\eps (S(\bT_r)^{-1})= \{ g \in \bG| g S(\bT_r)^{-1} g^\vdash \subset S(\bT_r)^{-1} \}; \]
it may be viewed as an algebraic group over $F$.

\begin{prop} The connected component $N_{\bG}^\eps(S(\bT_r)^{-1})^\circ$ of $N_{\bG}^\eps(S(T_r)^{-1})$ is equal to $T$. \end{prop}

\begin{proof} Note that $\bT \subseteq N_{\bG}^\eps(S(\bT_r)^{-1})^\circ$ by the Proposition \ref{tcent}.  

Suppose $g \in N_{\bG}^\eps(S(\bT_r)^{-1})$.  Then $ \eps(g) S(\bT_r) g^{-1} \subseteq S(\bT_r)$.  Applying $\nu$ we see that $ g (-\bT_r) g^{-1} \subseteq -\bT_r$.  Since $\bT_r$ is dense in $\bT$, we conclude that $g$ is in the usual normalizer $N_{\bG}(\bT)$.
Thus $N_{\bG}^\eps(S(\bT_r)^{-1}) \subseteq N_{\bG}(\bT)$.  From the theory of reductive groups we know that $N_{\bG}(\bT)^\circ=Z_{\bG}(\bT)^\circ=Z_{\bG}(\bT)$.  Thus $N_{\bG}^\eps(S(\bT_r)^{-1})^\circ \subseteq Z_{\bG}(\bT)$.

Now suppose $g \in N_{\bG}^\eps(S(\bT_r)^{-1})^\circ$ and let $t_1 \in \bT_r$.  Then there is an element $t_2 \in \bT_r$ so that 
$\eps(g)S(t_1)g^{-1}=S(t_2)$.  Taking norms gives $g(-t_1)g^{-1}=-t_2$.  Since $g \in Z_{\bG}(\bT)$ this implies that $t_1=t_2$.  Therefore $g \in {\bG}_{\eps,S(t_1)^{-1}}$, and therefore $g \in \bT$ by the previous proposition. 
\end{proof}

\begin{cor} The torus $\bT$ has finite index in $N_{\bG}^\eps(S(\bT_r)^{-1})$. \end{cor}

\begin{prop} \label{bfinite} The map 
\[ \beta: \bG/\bT \times \bT_r \ra \bG_{\eps rs} \] 
given by $\beta(g,\gm)= g S(\gm)^{-1} g^{\vdash}$ is a finite map. \end{prop}
That is to say, its fibers are finite.  Note this is well-defined by Proposition \ref{tcent}.
\begin{proof}

Suppose $\beta(g_1,\gm_1)=\beta(g_0,\gm_0)$.  Then $g_0^{-1}g_1 \in N_{\bG}^\eps(S(\bT_r)^{-1}),$ thus
$g_1$ ranges over the finite set $g_0 \cdot N_{\bG}^\eps(S(\bT_r)^{-1})/\bT$.  Since $g_0,g_1,$ and $\gm_0$ determine $\gm_1$, we are done. \end{proof}

\section{\bf Orbital Integrals}

For the reader's convenience we gather together a few facts on orbital integrals in this section.
The references are \cite{K} and \cite{KS}.

\begin{defn} For $\gm \in T_r$ let 
\[ D(\gm)= \det(\Ad(\gm)-1; \Lie(H)/\Lie(T)). \] \end{defn}

This is the usual Weyl discriminant.

\begin{comment}
\begin{prop} \label{IV} The quantity  $\kappa(\gm,S(\gm))=|D_{\eps}(S(\gm))|/|D(\gm)|$ is bounded above for all $\gm \in T_r$. \end{prop}
In fact, $\kappa$ is the transfer factor $\Delta_{IV}$ from \cite{KS}, and the proposition follows from Lemma 4.5.A of that paper.
\end{comment}

\begin{defn} Let $\phi(\gm)=\log_q \max \{1, |D(\gm)|^{-1} \}$.  \end{defn} 

\begin{defn} Let $\phi^S(\gm)=\log_q \max \{1, |D_\eps(S(\gm))|^{-1} \}$.  \end{defn} 

\begin{prop} \label{phis} 
There is an $\upsilon>0$ so that the function $|D(\gm) \cdot D_{\eps}(S(\gm))|^{-\upsilon}$ is locally integrable on $T$.
Given nonnegative integers $i,j$, the function $\phi^i \cdot (\phi^S)^j$ is locally integrable on $T$. \end{prop}

\begin{proof}  

Rather than generalizing Harish-Chandra's proof \cite{HC} of the corresponding facts for $D(\gm)$, we sketch a fancy proof, inspired by \cite{K}.

For $\gm \in \bT_r$, write $P(\gm)= D(\gm) \cdot D_{\eps}(S(\gm))$; it is a regular function in the sense of algebraic geometry.  Write $\Aff$ for affine space of dimension equal to $\rank(T)$. Given a point $t_0 \in \bT_r$ there is a rational open map $\varphi: \Aff \ra \bT$ with $t_0=\varphi(0)$ in the image.  In particular, if $t_0 \in T_r$, there are compact open neighborhoods $U$ of $t_0$ in $T(F)$ and $V$ of $0$ in $\Aff(F)$ so that the restriction of $\varphi$ to $V$ is a homeomorphism.  The map $P \circ \varphi$ is regular at $0$; we may assume it is regular on $V$.  
Pick a compactly supported function $\Phi$ on $\Aff(F)$ so that $\Phi dx=\xi_V \cdot \varphi^*(d\gm)$.

Then we have, for any complex number $s$,
\begin{equation}
\int_U |P(\gm)|^s d \gm = \int_{\Aff(F)} |P \circ \varphi(x)|^s \Phi(x)dx.
\end{equation}
We denote the expression on the right by $Z(s, \Phi)$, and turn to Igusa's study of this function in \cite{Ig}.
He only considers polynomial functions $f$, but his proof is valid for a rational function $f$ with no poles in $V$.

It is easy to see that his final expression for $Z(s,\Phi)$ converges for $\re(s)> \max \{ -\frac{\nu_i}{N_i} \}$, in his notation, a negative number.  In particular this converges for $s=-\upsilon$, for some $\upsilon>0$.  This proves the first part of the proposition.

The rest of the proposition follows from the following elementary fact:  For every $\upsilon >0$ and positive integers $i,j$ there is a constant $C$ so that
\[ \log \max \{ 1,y \}^i \cdot \log \max \{ 1,z \}^j \leq C \cdot (yz)^\upsilon. \]
\end{proof}

\begin{defn} For $\gm \in T_r$ and $f \in C_c^\infty(H)$, write $I(\gm,f)$ for the normalized integral of $f$ over the orbit of $\gm$.  That is,
\[  I(\gm,f)= |D(\gm)|^{\half} \int_{H/T} f(h \gm h^{-1})dh. \] \end{defn}

Remark:  Although it is possible to define these integrals for $\gm$ not regular, we do not do this.  We will extend $I(\cdot,f)$ to $T$ by $0$, keeping the same name, and do not want to confuse the reader.

\begin{defn} For $\gm \in T_r$ and $f \in C_c^\infty(H)$, let
\[  I^+(\gm,f)= |D(\gm)|^{\half} \int_{H^+/T} f(h \gm h^{-1})dh. \] \end{defn}

Note that if $w \in H^+ - H$ then

\[ I^+(\gm,f)=I(\gm,f)+I(\gm,f^w), \]
where $f^w(h)=f(whw^{-1})$.

The proofs of the next two propositions may be found in \cite{K}.

\begin{prop} \label{bound} The function $\gm \mapsto I(\gm,f)$ on $T_r$ is bounded and locally constant. \end{prop}

It is however not compactly supported on $T_r$.
Extend $I(\cdot,f)$ to $T$ by putting $I(\gm,f)=0$ at singular elements.  It is no longer locally constant.

\begin{prop} \label{comp} The function $\gm \mapsto I(\gm,f)$ on $T$ is compactly supported. \end{prop}

Along the way to proving these propositions is the following well-known result, which we will also use:

\begin{prop} \label{adj} Let $C$ be a compact subset of $H$.  The set $\{ t \in T| t$ is conjugate to an element of $C \}$ has compact closure. \end{prop}

\begin{defn} For $\delta \in \tilde{T}_r$ and $f \in C_c^\infty(G)$, let
\[I_{\eps}(\delta,f)= |D_\eps(\delta)|^{\half} \int_{G/G_{\eps,\delta}} f(g \delta g^{\vdash})dg. \] \end{defn}

The function $\gm \mapsto I_\eps(S(\gm)^{-1},f)$ is also bounded and locally constant in the sense of Propositions \ref{bound} and \ref{comp}.

\section{\bf Norms and Estimates}

Given an element $X \in M_n(F)$ write $|X|=\max |X_{ij}|$, the maximum taken over the entries of $X$.  This norm satisfies the relation $|XY| \leq |X||Y|$, and $|X|=0$ if and only if $X=0$. Write $\ord(X)=-\log_q(|X|)$.
Then one has $\ord(XY) \geq \ord(X)+\ord(Y)$.

\begin{defn} A lattice in $M_n(F)$ is an open compact subset. \end{defn}

Given a lattice $L \subset M_n(F)$ write $|L|=\max \{ |l| ; l \in L \}$ and $\ord(L)=-\log_q(L)$.  Let $L_0=M_n(\OO)$.
Write $\ord_*(L)= \min \{ i \in \Z; \varpi^i L_0 \subseteq L \}$. Note that $\ord(L)= \max \{ i \in \Z; L \subseteq \varpi^iL_0 \}$.  Thus $\ord_*(\varpi^i L_0)=\ord(\varpi^i L_0)=i,$ and in general $\ord(L) \leq \ord_*(L)$.

\begin{prop} Let $L$ be a lattice in $M_n(F)$ and $g \in G$. Then $\ord(gL) \geq \ord(g)+\ord(L)$, and $\ord_*(gL) \leq \ord_*(L)-\ord(g^{-1})$. \end{prop}

\begin{proof} We have $gL_0 \subseteq \varpi^{\ord(g)}L_0$ and $L \subseteq \varpi^{\ord(L)}L_0$.   It follows that
$gL \subseteq \varpi^{\ord(L)+\ord(g)}L_0$, whence the first statement.
Similarly, the inclusions $g^{-1}L_0 \subseteq \varpi^{\ord(g^{-1})}L_0$ and $\varpi^{\ord_*(L)}L_0 \subseteq L$ imply that $\varpi^{\ord_*(L)-\ord(g^{-1})}L_0 \subseteq gL$, whence the second statement.
\end{proof}

\begin{cor} With notation as in the previous proposition, we have $\ord(Lg) \geq \ord(g)+\ord(L)$, and $\ord_*(Lg) \leq \ord_*(L)-\ord(g^{-1})$. \end{cor}
\begin{proof} This follows by considering the transposes. \end{proof}

The following is obvious, but we will use this formulation later.

\begin{cor} If $L$ is a lattice in $M_n(F)$ and $g,h \in G$.  Then $\ord(gLh) \geq \ord(g)+\ord(h)+\ord(L)$, and $\ord_*(gLh) \leq \ord_*(L)-\ord(g^{-1})-\ord(h^{-1})$. \end{cor}

Given an element $g \in G$, write $||g||=\max \{ |g|,|\det(g)|^{-1} \}$.  Note that $||g||=\max \{ |g|,|g^{-1}| \}$.
Also for $g \in G$, write $||g||_{T \backslash G}=\inf_{t \in T} \{ ||tg|| \}$. 

\begin{prop} Let $C \subset G$ be a compact set.  Then there exist positive constants $c_1,c_2 >0$ so that for all $\gm \in T_r$ and all $g \in G$ so that $g S(\gm) g^{\vdash} \in C$, we have 

\[ \log ||g||_{T\backslash G} \leq c_1+c_2 \phi^S(\gm). \] \end{prop}

Recall that $\phi^S(\gm)=\log_q \max \{1, |D_\eps(S(\gm))|^{-1} \}$.

Our proof is almost identical to the proof of Lemma 20.3 in \cite{K}, but a few changes are necessary.

\begin{proof} 
Write $G_{\eps rs}$ for the set of $\eps$-regular, $\eps$-semisimple elements of $G$.
An element $\delta \in G$ is $\eps$-regular exactly when $D_\eps(\delta) \neq 0$.  (See Section 2 of \cite{C}.)
Therefore a norm (in the sense of \cite{K}, section 18.1) on $G_{\eps rs}$ is given by $||\delta||_{G_{\eps rs}}= \max \{ ||\delta||, |D_\eps(\delta)|^{-1} \}$.

Consider the morphism 
\[ \beta: G/T \times T_r \ra G_{\eps r s} \]
defined by $\beta(g,\gm)= g S(\gm)^{-1} g^{\vdash}$; in Proposition \ref{bfinite},  we showed that $\beta$ is finite.  We may therefore take (see Proposition 18.1 of [K]) as norm on 
$(G/T) \times T_r$ the pullback of $|| \cdot ||_{G_{\eps rs}}$ by $\beta$.  

By Proposition 18.1 of \cite{K} again, the pullback of the norm $|| \cdot ||_{G/T}$ to $(G/T) \times T_r$ (pull back using the first projection) is dominated by the norm on $(G/T) \times T_r$.  
(We are implicitly using the fact that the morphism $G \rightarrow G/T$ has the norm descent property, by Proposition
18.3 of \cite{K}.)
 This means, that there are constants $c>1$ and $R>0$ so that

\[ ||g||_{G/T} \leq c \max \{|| g S(\gm)^{-1} g^{\vdash}||, |D_\eps(S(\gm))|^{-1}| \}^R  \]
for all $g \in G/T$ and all $\gm \in T_r$.  

Since $C$ is compact, the restriction of $|| \cdot ||$ to $C$ is bounded above by some $d$.  Thus

\[ ||g||_{G/T} \leq cd^R \max \{1, |D_\eps(S(\gm))|^{-1} \}^R \]
for all $g \in G/T, \gm \in T_r$ such that $g \gm g^{-1} \in C$.  The proposition follows by taking the logarithm of both sides.
\end{proof}

We will prefer the following formulation later.
\begin{cor} Let $C \subset G$ be a compact set.  There exist positive constants $c_1,c_2 >0$ so that for all $\gm \in T_r$ and $gT \in G/T$ so that $g S(\gm)^{-1} g^{\vdash} \in C$, we may pick a representative $g_0 \in gT$ so that
\[ -\ord(g_0^{-1}) \leq c_1+c_2 \phi^S(\gm) \text{ and } \ord(g_0) \geq -(c_1+c_2 \phi^S(\gm)) \]
\end{cor}

Similarly, a more direct application of Lemma 20.3 of \cite{K} gives:
\begin{cor} Let $C \subset H$ be a compact set.  There exist positive constants $c_1,c_2 >0$ so that for all $\gm \in T_r$ and $Th \in T \backslash H$ so that $ h^{-1} \gm h \in C$, we may pick a representative $h_0 \in Th$ so that
\[ -\ord(h_0^{-1}) \leq c_1+c_2 \phi(\gm) \text{ and } \ord(h_0) \geq -(c_1+c_2 \phi(\gm))   . \]
\end{cor}

\section{\bf Volume estimation}

In this section $L$ denotes a general lattice which is stable under $\GL_n(\OO)$.  The application will be the lattice $L'$ from Section \ref{GS}.

Let $T$ be a torus in $H$ of split rank $r$.

\begin{defn} Given a lattice $L \subset M_n(F)$, matrices $g \in G$, $h \in H^+$, and $k \in \Z$, write $w_k(g,h)=w_k^L(g,h)$ for $\vol_T(T \cap \varpi^{-k}g^{-1}Lh^{-1})$. \end{defn}

Note that this is finite, and well-defined for $g \in G/T$ and $h \in T \backslash H^+$.
If $g$ and $h$ are fixed, then as $k$ grows, $w_k(g,h)$ increases to the volume of $T$, which is infinite unless $T$ is compact.

Such volumes play an important role in evaluating $R(f_G,f_H)$, and we estimate them in this section.

Write $G_L$ for the stabilizer of $L$ in $G=\GL_n(F)$; it is a compact open subgroup of $G$.  Consider the following assumption on our torus:

($\clubsuit$) $T$ can be written as the product $T=AT_c$ of a split torus $A$ and a compact torus $T_c$ with the property that $T_c \subseteq G_L$ and $A$ is the set of diagonal matrices of the form
\[ \diag(a_1,a_2, \cdots, a_r, 1, \cdots, 1, a_r^{-1}, \cdots, a_2^{-1}, a_1^{-1}), \]
with $a_i \in F^{\times}$.

First let us explicitly compute the quantity $w_0(g)=w_0(g,1)=\vol_T(T \cap g^{-1} L_0)$ where $L_0=M_n(\OO)$. 
If $T$ satisfies ($\clubsuit$), this is equal to $\vol_A(A \cap g^{-1} L_0)$.

\begin{defn} Given a vector $v=(c_1,\ldots,c_n) \in F^n,$ write $\ord(v)=\min \ord(c_i)$. 
\end{defn}

\begin{prop} \label{dilation}  Suppose $T$ satisfies ($\clubsuit$).  Write $v_1,\ldots,v_n$ for the columns of $g$.  

For $1 \leq i \leq r,$ let
 $\Delta_i(g)=\ord(v_{n-i})+\ord(v_i).$  Then
\[ w_0(g)=\begin{cases}
            0 & \text{ if some } \Delta_i(g)< 0  \\ 
            \prod_{i=1}^r (\Delta_i(g)+1) & \text{ if all }  \Delta_i(g) \geq 0\\
    \end{cases}.        \] \end{prop}

\begin{proof} The condition $a \in A \cap g^{-1} L_0$ exactly means that $ga$ has integral entries.

Let $a=\diag(a_1,a_2, \cdots, a_r, 1, \cdots, 1, a_r^{-1}, \cdots, a_2^{-1}, a_1^{-1})$.  The first $r$ columns of $ga$ are
$a_1v_1,\ldots a_rv_r,$ and the condition that these dilated columns are integral means that all $\ord(a_iv_i) \geq 0$. Thus we need $\ord(a_i) \geq -\ord(v_i)$.  
On the other hand the last $r$ columns of $ga$ are $a_r^{-1}v_{n-(r+1)},\ldots a_1^{-1}v_n,$ and the condition that these are integral means that $\ord(v_{n-1}) \geq \ord(a_i)$.  Thus $a \in A \cap g^{-1} L_0$ exactly when
$-\ord(v_i) \leq \ord(a_i) \leq \ord(v_{n-i})$.  This is impossible if $\Delta_i(g)$ is negative.  If $\Delta_i(g)$ is positive, there are $\Delta_i(g)+1$ different possible valuations for each $a_i$.  The proposition follows. \end{proof}

\begin{cor} \label{w_k} With the same notation as above,
\[ w_k(g)=\begin{cases}
            0 & \text{ if some } \Delta_i(g)< -2k  \\ 
            \prod_{i=1}^r (\Delta_i(g)+2k+1) & \text{ if all }  \Delta_i(g) \geq -2k\\
    \end{cases}.        \] \end{cor}

\begin{proof} We have $w_k(g)=w_0(\varpi^kg)$, and $\Delta_i(\varpi^kg)=\Delta_i(g)+2k$. \end{proof}

\begin{cor} \label{Moo} 
\[ \vol_T(T \cap \varpi^{-k} M_n(\OO)) = \begin{cases}
               (2k+1)^r & \text{ if } k \geq 0  \\
               0 & \text{ if } k < 0 \\
                \end{cases}.  \]
\end{cor}

\begin{cor} We have

\[ w_k(1,h)=\begin{cases}
            0 & \text{ if some } \Delta_i({}^th)< -2k  \\ 
            \prod_i (\Delta_i({}^th)+2k+1) & \text{ if all }  \Delta_i({}^th) \geq -2k\\
    \end{cases}.        \] \end{cor}

Similar reasoning to the proof of Proposition \ref{dilation} gives a lower bound for $w_k(g,h)$:

\begin{prop} If $L=L_0$ and $\Delta_i(g)+\Delta_i({}^th)+2k 
\geq 0$ for all $i$, then
\[ \prod_i (\Delta_i(g)+\Delta_i({}^th)+2k+1) \leq w_k(g,h). \]
\end{prop}

\begin{proof} 
The proposition reduces at once to the case where $k=0$.
Let $t=\diag(t_1,\ldots,t_m,t_m^{-1},\ldots,t_1^{-1})$.  Write $v_1,\ldots,v_n$ for the columns of $g$ and $w_1,\ldots,w_n$ for the rows of $h$.  Let $e_i=-\ord(v_i)$ and $f_i=-\ord(w_i)$.  Then the product
$gth=g't'h'$, where the columns of $g'$ are given by $v_i'=\varpi^{e_i}v_i$, the rows of $h'$ are given by 
$w_i'=\varpi^{f_i}w_i$, and $t'$ is the product of $t$ with $\diag(\varpi^{-e_1-f_1},\varpi^{-e_2-f_2},\ldots, \varpi^{-e_n-f_n})$.  Then $g',h'$ are integral, and $t'$ will be integral if and only if for $1 \leq i \leq m$,
we have $-e_i-f_i \leq \ord(t_i) \leq e_{n+1-i}+f_{n+1-i}$.  Thus for $\ord(t)$ in this range, the product $gth$ is integral.  Note that $e_i+e_{n+1-i}=\Delta_i(g)$ and $f_i+f_{n+1-i}=\Delta_i({}^th)$.  There are
$\Delta_i(g)+\Delta_i({}^th)+1$ possibilities for each $\ord(t_i)$ using this approach, and the estimate follows.
\end{proof}

\begin{prop} Continue to assume that $T$ satisfies ($\clubsuit$).  Fix compact sets $C_G \subset G$ and $C_H \subset H$.  Suppose $g \in G$ and $h \in H$, $\gm \in T_r$ with $g S(\gm)^{-1} g^{\vdash} \in C_G$ and $h^{-1} \gm h \in C_H$.  Let $L$ be a lattice in $M_n(F)$.  Then there are positive constants $c_1$, $c_2$, and $c_3$, depending only on $C_G, C_H,$ and $L$, so that the following two statements hold.  If $2k+c_1+c_2 \phi(\gm)+c_3 \phi^S(\gm) <0,$ then $ w_k(g,h)=0$.  If $2k-c_1-c_2 \phi(\gm)-c_3 \phi^S(\gm) \geq 0$ then
\[ (2k-c_1-c_2 \phi(\gm)-c_3 \phi^S(\gm))^r \leq w_k(g,h) \leq (2k+c_1+c_2\phi(\gm) + c_3 \phi^S(\gm))^r. \]

\end{prop}

Note that the last inequality is equivalent to
\[ |w_k(g,h)^{\frac{1}{r}}-2k| \leq c_1+c_2 \phi(\gm) + c_3 \phi^S(\gm). \]

\begin{proof}

By Corollary \ref{Moo}, we know that for any lattice $L$, $\vol_T(T \cap L)=0$ if $\ord(L)>0$, and if $\ord(L) \leq \ord_*(L) \leq 0,$
\[ (-2 \ord_*(L)+1)^r \leq \vol_T(T \cap L) \leq (-2 \ord(L)+1)^r . \]

We find the upper estimate first.  By the section on norms, we know that

\[ \ord(\varpi^{-k}g^{-1}L h^{-1}) \geq \ord(g^{-1})+\ord(L)+ \ord(h^{-1}) -k. \]

Thus
\[ \vol_T(T \cap \varpi^{-k}g^{-1}L h^{-1}) \leq (-2 [\ord(g^{-1})+\ord(L)+ \ord(h^{-1}) -k]+1)^r \]
Combining this with the upper estimates for $-\ord(g^{-1})$ and $-\ord(h^{-1})$ from the previous section gives
positive constants $c_1,c_2,c_3$ so that
\[ \vol_T(T \cap \varpi^{-k}g^{-1}L h^{-1}) \leq (2k +c_1+c_2 \phi(\gm)+c_3 \phi^S(\gm))^r. \]

Next, the lower estimate.  By the section on norms, we have

\[ \ord_*(\varpi^{-k}g^{-1}L h^{-1}) \leq \ord_*(L)-\ord(g)- \ord(h) -k. \]
Thus, 
\[ \vol_T(T \cap \varpi^{-k}g^{-1}L h^{-1}) \geq (2 [\ord(g)-\ord_*(L)+ \ord(h) +k]+1)^r. \]
Combining this with the lower estimates for $\ord(g)$ and $\ord(h)$ from the previous section gives positive constants
$c_1',c_2',c_3'$ so that
\[ \vol_T(T \cap \varpi^{-k}g^{-1}L h^{-1}) \geq (2k -c_1'-c_2' \phi(\gm)-c_3' \phi^S(\gm))^r. \]

The result follows.

\end{proof}

We now extend part of this for the general maximal torus $T \subseteq H$.

\begin{cor} \label{vol} Let $T$ be any maximal torus of $H$, with split rank $r$.  Fix compact sets $C_G \subset G$ and $C_H \subset H$.  Suppose $g \in G$ and $h \in H$, $\gm \in T_r$ with $g S(\gm)^{-1} g^{\vdash} \in C_G$ and $h^{-1} \gm h \in C_H$.  Let $L$ be a lattice in $M_n(F)$.  Then there are positive constants $c_L$, $c_1$, $c_2$, and $c_3$, depending only on $C_G, C_H,$ and $L$, so that the following two statements hold.  If $2k+c_1+c_2 \phi(\gm)+c_3 \phi^S(\gm) <0,$ then $ w_k(g,h)=0$.  If $2k+c_1+c_2 \phi(\gm) + 
c_3 \phi^S(\gm) \geq 0$ then
\[  w_k(g,h) \leq c_L (2k+c_1+c_2\phi(\gm) + c_3 \phi^S(\gm))^r. \]

\end{cor}
\begin{proof} By conjugating $T$ we may assume it may be written as a product $T=AT_c$ with $A$ as in ($\clubsuit$) and $T_c$ compact.  The intersection $T_{c,L}$ of $G_L$ with $T_c$ has finite index inside $T_c$, and therefore the product $T_L=AT_{c,L}$ has finite index $\ell$ inside $T$.  Write $x_1, \ldots, x_\ell$ representing the quotient.  Then one has 
\[ w_k(g,h)=\sum_i w_k^L(gx_i,h), \]
where $w_k^L(g,h)$ is computed relative to the torus $T_L$, which satisfies ($\clubsuit$).
Each of the terms in the sum satisfies an upper estimate as in the previous proposition, and we may take $c_L=\ell$.
\end{proof}

\section{\bf Absolute Integrality}

Choose once and for all a set $A$ of representatives for $F^\times / F^{\times 2}$.  

For $k \in \Z$, write $Z_k= \{z \in Z(G); |z|=q^k \}.$
Say $\vol_{Z(G)}(Z_k)=1$ for all $k$.

We may write the quantity $\psi(s,\gm)$ as

\[ \sum_{\alpha \in A} \omega(\alpha)^{-1} \int_{G/T} \int_{T \backslash H^+} f_G(\alpha \cdot gS(\gm)^{-1}g^{\vdash}) f_H(h^{-1} \gm h) |\det(gS(\gm)^{-1}g^{\vdash})|^s \cdot  \sum_{k \in \Z} q^{-2nks} \cdot \tilde{w}_k(g,h) dh dg,\]
where
\[\tilde{w}_k(g,h)=\int_T \int_{Z_k} \xi_{L^\vdash}(z^{-2}\cdot gS(\gm)^{-1}g^{\vdash})\xi_{L'}(z^{-1}gth)dzdt. \]
 
\begin{lemma} Suppose $L'$ is $\OO^{\times}$-invariant.  Then for all $g \in G$, $h \in H$ and $k \in \Z$, $\tilde{w}_k(g,h) \leq w_k(g,h)$.
\end{lemma}
\begin{proof}
In fact, $w_k(g,h)=\int_T \int_{Z_k} \xi_{L'}(z^{-1}gth)dzdt.$ \end{proof}

Remark:  In general, $L'$ is contained in some $L_i=\varpi^{-i} M_n(\OO)$, which is $\OO^{\times}$-invariant.  Then,
\[ \tilde{w}^{L'}_k(g,h) \leq  \tilde{w}^{L_i}_k(g,h) \leq w_k^{L_i}(g,h). \]
Therefore the convergence results in this section are true for all lattices, but one must modify the definition of $w_k(g,h)$ accordingly to generalize Theorem \ref{R}

Note that $\supp f_G$ is compact and does not contain $0$.  The set $L^\vdash$ is compact, and the set $A$ is finite. Therefore there is a $k_-$ so that if $k<k_-$ and $\alpha \cdot gS(\gm)^{-1}g^{\vdash} \in \supp f_G$, then $z^{-2} \cdot gS(\gm)^{-1}g^{\vdash} \notin L^\vdash$ for $z \in Z_k$.
Therefore $\tilde{w}_k(g,h)$ vanishes for such $k$, and we deduce that
\[ \psi(s, \gm)= \sum_{\alpha \in A} \omega(\alpha)^{-1} \int_{G/T} \int_{T \backslash H^+} f_G(\alpha \cdot gS(\gm)^{-1}g^{\vdash}) f_H(h^{-1} \gm h) |\det(gS(\gm)^{-1}g^{\vdash})|^s \cdot  \sum_{k \geq k_-} q^{-2nks} \tilde{w}_k(g,h) dh dg.\]

\begin{prop} For $\re(s)>0$, the integral $\int_T \psi(s,\gm)|D_\eps(\gm)| d\gm$ converges absolutely. \end{prop}

\begin{proof} Let $M=\max \{|\det(\alpha x)|^{\re(s)}; x \in \supp(f_G), \alpha \in A \}< \infty$. 
By the above lemma and Corollary \ref{vol}, we may use the estimates $\tilde{w}_k(g,h) \leq w_k(g,h) \leq (2k+c_1+c_2 \phi(\gm)+c_3 \phi^S(\gm))^r$ for positive constants $c_1,c_2,c_3$.  By expanding the $r$th power, we reduce to proving that, for nonnegative integers $j_1,j_2$, and $i$, the expression
\[ \int_T \int_{G/T} \int_{T \bks H^+}|D_\eps(\gm)|  f_G(\alpha \cdot gS(\gm)^{-1}g^{\vdash}) f_H(h^{-1} \gm h) \cdot M \cdot \phi(\gm)^{j_1} \phi^S(\gm)^{j_2} \sum_{k \geq k_-} q^{-2nks} (2k)^i dh dg d\gm \]
converges absolutely.  The sum is independent of $\gm$ and converges absolutely for $\re(s)>0$.
The rest of the integral is
\[ \int_T  I_{\eps}(S(\gm)^{-1},f_G(\alpha \cdot))I^+(\gm,f_H) \cdot M \cdot  \phi(\gm)^{j_1} \phi^S(\gm)^{j_2}   \cdot  d\gm .\]

By Propositions \ref{phis}, \ref{bound}, and \ref{comp}, this is absolutely integrable.\end{proof}

We will use this result to switch around the sum over $k$ when convenient.

For example, define 
\[ \tilde{\psi}_k(s, \gm)=  \sum_{\alpha \in A} \omega(\alpha)^{-1} \int_{G/T} \int_{T \backslash H^+} f_G(\alpha \cdot gS(\gm)^{-1}g^{\vdash}) f_H(h^{-1} \gm h) |\det(gS(\gm)^{-1}g^{\vdash})|^s \cdot \tilde{w}_k(g,h) dh dg,\]
and $\tilde{\Psi}_k(s)=\int_T \tilde{\psi}_k(s,\gm)|D_\eps(\gm)| d\gm$.

Then $\psi(s,\gm)= \sum_{k \geq k_-} q^{-2nks}\tilde{\psi}_k(s, \gm)$ and $R(f_G,f_H)=\Res_{s=0} \sum_{k \geq k_-} q^{-2nks} \tilde{\Psi}_k(s)$.

\begin{prop} \label{entire} For all $\gm \in T_r$, the function $\tilde{\psi}_k(s,\gm)$ is entire.  The function $\tilde{\Psi}_k(s)$ is also entire. \end{prop}

\begin{proof}

To prove $\tilde{\psi}_k(s,\gm)$ is entire it is enough to observe that for all $\alpha \in A$, the function $\varphi: \C \times (G/T \times T \bks H^+) \ra \C$ given by
\[ \varphi(s,g,h)=f_G(\alpha gS(\gm)^{-1}g^\vdash)f_H(h^{-1}\gm h)|\det(\alpha gS(\gm)^{-1}g^\vdash)|^s \tilde{w}_k(g,h)       \]
satisfies the conditions of Lemma \ref{Morera} below.  It is obviously entire for fixed $g,h$.

Let $K$ be a compact subset of $\C$.  The functions $f_G$ and $f_H$ are bounded above.  We may assume $gS(\gm)^{-1}g^{\vdash}$ is in the compact support of $f_G$, so that $|\det(\alpha gS(\gm)^{-1}g^\vdash)|^s$ is bounded above for $s \in K$.  Corollary \ref{vol} gives a bound for $w_k(g,h) \geq \tilde{w}_k(g,h)$ which only depends on $\gm$.  Moreover if we fix $s_0 \in \C$ the support of $\varphi(s_0,g,h)$ is compact by Proposition \ref{adj}.  Lemma \ref{Morera} then shows that $\tilde{\psi}_k(s, \gm)$ is entire.

Again let $K \subset \C$ be compact, fix $k$, and consider the function $\varphi: K \times T \ra \C$ given by $\varphi(s,\gm)= \tilde{\psi}_k(s, \gm)|D_\eps(\gm)|.$  We again employ Lemma \ref{Morera}.  By the above paragraph, $\varphi(s,\gm)$ is entire.  Using Corollary \ref{vol} again we note that 
\[|\varphi(s,\gm)| \leq g(\gm)=|I_\eps(S(\gm)^{-1},f_G)||I^+(\gm,f_H)|\cdot M \cdot c_L (2k+c_1+c_2\phi(\gm) + c_3 \phi^S(\gm))^r, \]
where $M=\max \{|\det(\alpha x)|^{\re(s)}; x \in \supp(f_G), \alpha \in A, s \in K \}$.  By Propositions \ref{bound} and \ref{comp}, the orbital integrals are bounded, and have compact support in $T$.
We may expand $c_L (2k+c_1+c_2\phi  + c_3 \phi^S )^r$ into the sum of a constant term $c_1^r$, and constant multiples of nonzero powers of $\phi$ and $\phi^S$.  By Proposition \ref{phis}, these nonzero powers are integrable on $T$.  It follows that $g$ is integrable on $T$.  Therefore $\tilde{\Psi}_k(s,\gm)$ is entire.

\end{proof}

\begin{lemma} \label{Morera} Let $(X,dx)$ be a $\sigma$-finite measure space.
Let $\varphi: \C \times X \ra X$ be a function, and $g \in L^1(X)$ so that
\begin{itemize}
\item For all $s \in \C$ the function $x \mapsto \varphi(s,x)$ on $X$ is measureable.
\item Given $x \in X,$ the function $\varphi(s,x)$ is entire.
\item Given a compact subset $K \subset \C$, there is a function $g \in L^1(X)$ so that $|\varphi(s,x)| \leq g(x)$ for $s \in K$ and $x \in X$.
\end{itemize}
Then the function $\Phi(s)= \int_X \varphi(s,t)dx$ is entire. \end{lemma}

\begin{proof} This is an application of Morera and Fubini's theorems.  Triangles are compact. \end{proof}

\begin{cor}  For any $k_* \in \Z$, the expression
\[  \Res_{s=0} \int_T \psi(s,\gm)|D_\eps(\gm)|d\gm \]
is equal to
\[  \Res_{s=0} \sum_{k \geq k_*} \int_T  \tilde{\psi}_k(s,\gm)|D_\eps(\gm)|d \gm. \]
\end{cor}

By Remark 6 of \cite{Sh-A}, we may pick $k_0$ so that for $k \geq k_0$, we have
\[\tilde{w}_k(g,h)= \int_T \int_{Z_k} \xi_{L'}(z^{-1}gh_0h)dh_0dz=w_k(g,h). \]

Define $\psi_k(s, \gm)= $
\[ \sum_{\alpha \in A} \omega(\alpha)^{-1} \int_{G/T} \int_{T \backslash H^+} f_G(\alpha \cdot gS(\gm)^{-1}g^{\vdash}) f_H(h^{-1} \gm h) |\det(gS(\gm)^{-1}g^{\vdash})|^s \cdot w_k(g,h) dh dg\]
and 
\[ \Psi_k(s)=\int_T \psi_k(s,\gm)|D_\eps(\gm)| d\gm. \]

One may prove that $\psi_k(s,\gm)$ and $\Psi_k(s)$ are holomorphic as in Proposition \ref{entire} (in fact it is easier).

We deduce the following.

\begin{prop}  
\[R(f_G,f_H)=\Res_{s=0} \sum_{k \geq 0}q^{-2nks}\int_T \psi_k(s,\gm)|D_\eps(\gm)|d \gm. \]
\end{prop}

We make a change of variables to absorb the $\alpha$ into the $w_k(g,h)$:

Given $\alpha \in F^{\times},$ let
\[ x_{\alpha}=\left( \begin{array}{cc}
\alpha I & 0 \\
0 & I \\
\end{array} \right) \in G,\]
where $I$ is the identity matrix of size $m$.

\begin{defn} Let $W_k(g,h,s)= \sum_{\alpha \in A} \omega(\alpha)^{-1} |\alpha|^{-ns} w_k(gx_{\alpha}^{-1},h)$.

Also let 
\[ W_k(g,h)=W_k(g,h,0)=\sum_{\alpha \in A} \omega(\alpha)^{-1} w_k(gx_{\alpha}^{-1},h).\]
 \end{defn}

This depends on the choice of $A$.

\begin{prop} \label{W} For $k \geq k_0,$
\[ \psi_k(s,\gm)= \int_{G/T} \int_{T \backslash H^+} f_G(gS(\gm)^{-1}g^{\vdash}) f_H(h^{-1} \gm h)|\det(gS(\gm)^{-1}g^{\vdash})|^s W_k(g,h,s)dh dg. \]\end{prop}

\begin{proof}

This is just the substitution $g'=gx_{\alpha}$.  Note that $f_G(\alpha \cdot {}^gS(\gm)^{-1})=f_G({}^{g'} S(\gm)^{-1})$. That there is no change of measure follows from the following lemma. \end{proof}

\begin{lemma} Let $f \in C_c^\infty(G/T)$.  Write $A$ for a torus containing $T$, and let $a_0 \in A$.  Then $\int_{G/T} f(ga_0)\frac{dg}{dt}=\int_{G/T} f(g) \frac{dg}{dt}$. \end{lemma}

\begin{proof} Since $G$ and $A$ are unimodular, one has a quotient measure $\frac{dg}{da}$ on $G/A$.  Then
\[ \int_{G/T} f(ga_0)\frac{dg}{dt}=\int_{G/A} \int_{A/T} f(gaa_0) da \frac{dg}{da}=\int_{G/A} \int_A f(ga) da \frac{dg}{da}=
\int_{G/T} f(g)\frac{dg}{dt}. \] \end{proof}

\begin{thm} \label{R}
\[ R(f_G,f_H)=\sum_T |W(T)|^{-1} \Res_{s=0}  \sum_{k=0}^{\infty} q^{-2nks} \int_T |D_\eps(\gm)|  \int_{G/T} \int_{T \backslash H^+}  f_G(gS(\gm)^{-1}g^{\vdash}) f_H(h^{-1}\gm h)  W_k(g,h) dh dg d\gm. \] \end{thm}

Recall that the first sum is over conjugacy classes of maximal tori $T$ in $H$.

\begin{proof} The factors $|\alpha|^{-ns}$ in the definition of $W_k(g,h,s)$ are holomorphic at $s=0$, do not depend on $k$, and factor out of the integral and sum.  By the following lemma, we may replace the quantity $|\det(gS(\gm)^{-1}g^{\vdash})|^s$ with $1$.
\end{proof}

\begin{lemma}
Let $X$ be a measure space, and $F: X \ra \C$ be a measureable function taking on only finitely many nonzero values. Let $g_k(s,x)$ be a sequence of functions on $\C \times X$ with $\sum_k \int_X g_k(s,x)dx$ absolutely integrable and with $g_k(s,x)$ holomorphic for fixed $k,x$ for $\re(s) >0$.
Then
\[ \Res_{s=0} \sum_{k \geq 0} \int_X |F(x)|^s g_k(s,x) dx =\Res_{s=0} \sum_{k \geq 0} \int_X g_k(s,x) dx. \]
\end{lemma}
\begin{proof} Say $y_1, \ldots y_r$ are the finitely many values of $F$.  Let $X_i=F^{-1}(y_i)$. 
Then for each $k$,
\[ \int_X |F(x)|^s g_k(s,x) dx=\sum_{i=1}^r |y_i|^s  \int_{X_i} g_k(s,x) dx. \]
The factor $|y_i|^s$ is holomorphic and factors out of the sum over $k$. \end{proof}

\section{\bf Supercuspidal representations of $\GL_2(F)$}

\subsection{The Function $f_G$} \label{f_G}
We wish to do some explicit calculations, for which we need formulas for various $f_G$.  This requires formulas for matrix coefficients of supercuspidal representations of of $\GL_2(F)$.

We will use a construction in Kutzko's paper \cite{Ku}.  In this paper he constructs irreducible supercuspidal representations $\pi_G=\mathcal{T}(\rho,\lambda,n,\alpha)$, via compact induction of characters on compact mod center subgroups $G'$ of $G$.
Proposition 2.5 of \cite{Ku} states that every irreducible supercuspidal representation of $\GL_2(F)$ is either unramified with quasiconductor $\pp$ or equivalent to some $\mathcal{T}(\rho,\lambda,n,\alpha)$.

We require a few facts about the ``inducing subgroup'' $G'$.  It is equal to $E^{\times}\KL$, where $E$ is a quadratic extension of $F$ and $\KL$ is a certain compact open subgroup.  We will take $E$ to be of the form $E=F(\sqrt{\varpi})$.
%($\star$ Is this necessary?) 
We note that in all cases of \cite{Ku}, $L$ is contained in
\[ I_1=\left( \begin{array}{cc} 
1+\pp & \OO \\
\pp & 1+\pp \\
\end{array} \right). \]

Also part of the data of $\pi_G=\mathcal{T}(\rho,\lambda,n,\alpha)$ is a certain character $\lambda$ of $E^\times$.
 
In fact $\lambda$ may be extended to $G'$, and compactly induces to our representation $\pi_G$.

Let $\omega=\lambda|_F$; this is the central character of $\pi_G$.  Since we restrict our attention to self-dual $\pi$, we have $\omega^2=1$.

Then a particular matrix coefficient of $\pi_G$, which we will denote by $\psi$, is given by extending the character $\lambda$ by $0$ to $G$. 

We need a compactly supported function $f_G$ so that

\[ \int_{Z(G)}\omega(z)^{-1} f_G(zg)dz=\psi(g). \]

\begin{lemma} Let $G$ be a locally compact group, and $Z$ the center of $G$.  Let $\omega: Z \ra \C$ be a character, and $\psi: G \ra \C$ a function so that $\psi(zg)=\omega(z)\psi(g)$ for all $z \in Z$.
Suppose $G'=\supp \psi$ is a subgroup, and $C_0 \subset G'$ is a closed subgroup.
Let $\{ t \}$ be a system of coset representatives for $G'/C_0Z$.
Put $C= \bigcup_t C_0 t,$ and $f= \psi \cdot 1_C$.

Then for all $g \in G,$

\[ \int_{Z(G)}\omega(z)^{-1} f(zg)dz=\psi(g) \cdot \vol_Z(C_0 \cap Z). \]
\end{lemma}

\begin{proof} If $g \notin G',$ then $zg \notin G',$ so both sides are $0$.
If $g \in G',$ then a simple computation shows
\[ \int_{Z(G)}\omega(z)^{-1} f(zg)dz=\psi(g) \cdot \vol_Z(Cg^{-1} \cap Z). \]

Claim: for all $h \in G', \vol_Z(Cg^{-1} \cap Z)=\vol(C_0 \cap Z).$
Note that $Ch=\bigcup_s C_0 \cdot s$, with $s=th$ a system of coset reps for $G'/C_0Z$.
Then $C_0s \cup Z \neq \phi \Leftrightarrow s \in C_0Z,$ which happens for exactly one $s_0=c_0z_0$.

Then it is easy to see that $\vol_Z(C_0s_0 \cap Z)=\vol_Z(C_0 \cap Z)$.  This proves the claim, and the lemma.
\end{proof}

To adapt this to our case, we have $G'=E^{\times} \KL$ and put $C_0=\OO_E^{\times} \KL.$ 
Note that $G'/C_0= E^{\times}/F^{\times}\OO^{\times}= \{ 1, \varpi_E \}$ since $E$ is totally ramified.

We therefore set $f_G$ to be the product of $\psi$ with the characteristic function $1_C,$ where $C$ is the compact subset $\OO_E^{\times} \cdot \KL \cdot \{1,\varpi_E \}$.  Write $C_0=\OO_E^{\times} \cdot \KL.$  Note that $C_0 \subset I_0,$ where $I_0$ is the parahoric subgroup of $G$ defined as:

\[ I_0=\left( \begin{array}{cc} 
\OO^{\times} & \OO \\
\pp & \OO^{\times} \\
\end{array} \right). \]

Break $f_G$ into functions $f_G=f_0+f_1,$ where $f_0=\psi \cdot \xi_{C_0}$ and 
$f_1=\psi \cdot \xi_{\varpi_E \cdot C_0}$.  (Recall that $\xi_S$ denotes the characteristic function of $S$.)

\subsection{Sorting Out the Integral}

We aim to compute the quantity $R(f_G,f_H)$ for various choices of $\pi_G$.  We have $G=\GL_2(F)$ and $H=T$, the split $\SO(2)$.

For simplicity we will take $\pi_H$ to be the one-dimensional trivial representation of $H$.  Thus $f_H(\gm)=1$.  We will put $L'=L_0=M_2(\OO)$.

Note that $T \bks H^+= \{ 1, w \}$, with $w=\left( \begin{array}{cc}
0 &1 \\
1& 0 \\
\end{array} \right)$, and therefore

\[ \int_{G/T} \int_{T \bks H^+} f_G(gS(\gm)^{-1}g^{\vdash}) f_H(\gm^h)W_k(g,h)dh \frac{dg}{dt} = \]
\[ \int_{G/T} f_G(gS(\gm)^{-1}g^{\vdash}) (W_k(g,1)+  W_k(g,w)) \frac{dg}{dt}. \]

and since $w \in K,$ we have $W_k(g,w)=W_k(g,1)$.

\[\text{Let } \tilde{I}(s,f_G,f_H)=  \sum_{k=0}^{\infty} q^{-2nks} \int_T |D_\eps(\gm)|  \int_{G/T} \int_{T \backslash H^+}  f_G(gS(\gm)^{-1}g^{\vdash}) f_H(h^{-1}\gm h)  W_k(g,h) dh dg d\gm, \]
as in the statement of Theorem \ref{Thm1}.

By the above computations, we have

\begin{prop} \label{sort}
\[ \tilde{I}(s,f_G,f_H)= 2 \sum_{k=0}^{\infty} q^{-4ks} \int_T |D_\eps(\gm)|\int_{G/T} f_G(gS(\gm)^{-1}g^{\vdash}) W_k(g) \frac{dg}{dt} d\gm. \]\end{prop}

Here, $W_k(g)=W_k(g,1)$.

Corollary \ref{w_k} leads to the following exact formula for $W_k(g)$.  Recall that if $g= \left( \begin{array}{cc}
a &b \\
c& d \\
\end{array} \right),$ then $\Delta_1(g)=\min \{ \ord(a),\ord(c) \}+ \min \{ \ord(b),\ord(d) \}$.

\begin{prop} 

\[ W_k(g)=\begin{cases}
            0 & \text{ if } \Delta_1(g)< -2k  \\ 
            
            |\OO^\times/\OO^{\times 2}| (2\Delta_1(g)+4k+1) & \text{ if }  \Delta_1(g) \geq -2k\\
    \end{cases}.        \] \end{prop}

\section{Calculation of $\psi_k$}

At this juncture we write $f$ rather than $f_G$ for simplicity.

We will soon compute the integrals
\[ \psi_k(\gm)=\int_{G/T} f(g S(\gm)^{-1}g^{\vdash})W_k(g)\frac{dg}{dt} \]
in cases of interest to us.

We write $\gm \in T$ as $\left( \begin{array}{cc} 
\alpha & 0  \\ 
0 &  \alpha^{-1} \\
 \end{array} \right)$, with $\alpha \in F^{\times}$, and $\alpha \neq \pm 1$.

Note that $S(\gm)=\left( \begin{array}{cc} 
\alpha-1 & 0  \\ 
0 &  \alpha^{-1}-1 \\
 \end{array} \right)$.

\subsection{Twisted Conjugacy}

As we will see, many of the integrals $\psi_k(\gm)$ vanish simply because no $\eps$-conjugates of $S(\gm)^{-1}$ meet the support of $f$.

We recall that $C=\supp(f)$ can be written as $C_0 \cup \varpi_E C_0$ with $C_0= \OO_E^{\times} \KL$, and $\KL \subseteq I_1$.  Note that $C_0 \subset I_0 \subset K=\GL_2(\OO)$.

In particular, elements of $C_0$ are upper triangular mod $\pp$.  They are also $\eps$-symmetric mod $\pp$, in the following sense.
 
\begin{defn} Say $X$ is $\eps$-symmetric if $X^\vdash=X$.
\end{defn}

\begin{lemma} \label{easy} If $X$ is $\eps$-conjugate to $Y$, and $X$ is $\eps$-symmetric, then $Y$ is $\eps$-symmetric. \end{lemma}

\begin{proof} This is easy. \end{proof}

\begin{prop} \label{abig} If $\alpha \notin \OO^{\times}$ then $S(\gm)^{-1}$ is not $\eps$-conjugate to any element of $C$. \end{prop}

Therefore $\psi_k(\gm)=0$ for noncompact elements of $T$.

\begin{proof} 

The proposition is equivalent to showing that for such $\gm$, $S(\gm)$ is not $\eps$-conjugate to any element of
$C_0 \cup C_0 \varpi_E^{-1}$.

We will be using the Iwasawa decomposition in what follows .  Thus, write elements $g \in G$ as $g=\kappa n_ba_it$ with $\kappa \in K$,
$n_b= \left( \begin{array}{cc} 
1 & b  \\ 
0 &  1 \\
 \end{array} \right)$, $a_i=\left( \begin{array}{cc} 
\varpi^i & 0  \\ 
0 &  1 \\
 \end{array} \right)$, and $t \in T$
.  We may assume $t=1$.

Assume first that $\ord(\alpha)$ is a positive even number.  Then $\ord(\det(S(\gm)))=\ord(\alpha^{-1})=-2e$.
In this case $S(\gm)$ cannot be $\eps$-conjugate to $C_0 \varpi_E^{-1}$, and we must take $i=e$.
However, we have
\[ n_ba_e S(\gm) (n_ba_e)^\vdash=\left( \begin{array}{cc} 
\varpi^e(\alpha-1) & *  \\ 
0 &  \varpi^e(\alpha^{-1}-1) \\
 \end{array} \right). \]
 
These diagonal element are not units, thus this not an element of $K \supset C_0$.  We conclude that $n_ba_i S(\gm) (n_ba_i)^\vdash$ is not $\eps$-conjugate to $C_0$ by an element of $K$.  So we are done with the case where $\ord(\alpha)$ is a positive even integer.

Now assume that $\ord(\alpha)$ is a positive odd number, say $2e+1$.  Then a similar computation to the above shows first that $S(\gm)$ cannot be $\eps$-conjugate to $C_0$, and we must take $i=e$.  One computes
\[ n_ba_i S(\gm) (n_ba_i)^\vdash=\left( \begin{array}{cc} 
\varpi^e(\alpha-1) & b \varpi^e(\alpha+\alpha^{-1}-2)   \\ 
0 &  \varpi^e(\alpha^{-1}-1) \\
 \end{array} \right). \]
If this is $\eps$-conjugate to $C_0 \varpi_E^{-1}$ by an element of $K$, then it must be an element of 
$K\varpi_E^{-1}K$.  The entries of such a matrix may not have an $\ord$ less than $-1$.  Now $\ord(\varpi^e(\alpha^{-1}-1))=-e-1$, so the only possibility is that $e=0$.
Next, note that $\ord(b(\alpha+\alpha^{-1}-2))=\ord(b)-1$, so we may take $b=0$ (recall that we may assume $b \in F/\OO$).  Thus $n_ba_i=1$ and we reduce to the question of whether $S(\gm)$ itself is $\eps$-conjugate to $C_0 \varpi_E^{-1}$ by an element of $K$.  

Let $\kappa=\left( \begin{array}{cc} 
a & b  \\ 
c & d \\
 \end{array} \right) \in K$.  We will show that it is impossible that $\kappa S(\gm) \kappa^{-1} \varpi_E \in \OO_E^{\times}I_1$, which will complete the proposition for $\ord(\alpha)=1$.
One computes 
\[ \kappa S(\gm) \kappa^{\vdash}=\left( \begin{array}{cc} 
ad(\alpha-1)+bc(\alpha^{-1}-1) & *  \\ 
* & bc(\alpha-1)+ad(\alpha^{-1}-1) \\
 \end{array} \right). \]
and therefore
\[ \kappa S(\gm) \kappa^\vdash \varpi_E=\left( \begin{array}{cc} 
* & ad(\alpha-1)+bc(\alpha^{-1}-1)  \\ 
bc \varpi (\alpha-1)+\varpi ad(\alpha^{-1}-1) & * \\
 \end{array} \right). \]
For this to be integral requires that $bc \in \pp$.  Since $E=F(\sqrt{\varpi})$, $\OO_E^\times I_1$ is upper diagonal mod $\pp$.  This forces $ad \in \pp$.  But these two conditions imply that $\det \kappa=ad-bc \in \pp$. which is a contradiction.  So we are done with the case that $\ord(\alpha)$ is a positive odd number.

Next, suppose that $\alpha^{-1} \in \pp$.  Then we have $w S(\gm) w^\vdash=\left( \begin{array}{cc} 
\alpha^{-1}-1 & 0  \\ 
0 & \alpha-1 \\
 \end{array} \right),$ which reduces us to the case of $\alpha \in \pp$, which we have already ruled out.
\end{proof}

Thus we may assume $\alpha \in \OO^{\times}$. 

\begin{prop} \label{aother} Let $\alpha \in \OO^{\times}$. If $\alpha \neq \pm 1$ mod $\pp$, then $S(\gm)^{-1}$ is not $\eps$-conjugate to any element of $C$. \end{prop}

\begin{proof} Note that in this case $S(\gm) \in K$, so cannot be $\eps$-conjugate to an element of $C_0 \varpi_E^{-1}$.  Moreover    we must have $a_i=1$.  We compute
\[ n_b S(\gm) n_b^\vdash=  \left( \begin{array}{cc} 
\alpha-1 & * \\ 
0 & \alpha^{-1}-1 \\
 \end{array} \right). \]
Suppose $b$ is chosen so that this is in $K$.
Note that any element of $K$ which is $\eps$-conjugate to an element of $C_0$ must itself be $\eps$-symmetric mod $\pp$.  This forces $\alpha \equiv \pm 1$ mod $\pp$, a contradiction.
\end{proof}
 
\begin{prop} \label{symmetric} Suppose the residue characteristic is odd, and $\alpha =1$ mod $\pp$.  Then $S(\gm)^{-1}$ is not $\eps$-conjugate to any element of $C$. \end{prop}

\begin{proof} Write $\alpha =1+x,$ and say $\ord(x)=e$.  Then 

\[ S(\gm)^{-1}= \left( \begin{array}{cc} 
x^{-1} & 0 \\ 
0 & -x^{-1}-1 \\
 \end{array} \right). \]

Note that $\ord(\det(S(\gm)^{-1}))=2e,$ so we reduce to showing $S(\gm)^{-1}$ is not $\eps$-conjugate to any element of $C_0$.
Again, writing $g=\kappa n_b a_i$, we must have $i=e$.  In this case, 

\[ {}^p S(\gm)^{-1}=\left( \begin{array}{cc} 
\varpi^ex^{-1} &-b\varpi^e \\ 
0 & \varpi^e(-x^{-1}-1) \\
 \end{array} \right), \]

This is in $K$ if and only if $\ord(b) \geq -e$.

If this were symmetric $\Mod \pp$ we would have $2\varpi^ex^{-1}=0 \Mod \pp$, which is impossible in odd residue characteristic. Thus by Lemma \ref{easy}  we are done.

\end{proof}

The last case to study in the odd residue characteristic is $\alpha \equiv -1 \Mod \pp$.

\begin{thm} Suppose the residue characteristic is odd.  Then 
\[ \tilde{I}(s,f_G,1)= |\OO^\times/\OO^{\times 2}|\left( \sum_{k=0}^\infty q^{-4ks}(4k+1) \right) R_G(f_G,1), \]
where $R_G(f_G,1)$ is given by
\[ \int_T |D_\eps(\gm)| \int_K f(\kappa S(\gm)^{-1} \kappa^\vdash)d \kappa d \gm. \]
\end{thm}

Note that $R_G(f_G,1)$ does not depend on $k$; the weight factor plays no role here.

\begin{proof}
We are left with considering the case $\alpha =-1$ mod $\pp$.
Proceeding as in the proof of Proposition \ref{aother}, we conclude again that we must have $a_i=1$.

In this case, 

\[ n_b S(\gm) n_b^\vdash=  \left( \begin{array}{cc} 
\alpha-1 & (\alpha+\alpha^{-1})b \\ 
0 & \alpha^{-1}-1 \\
 \end{array} \right). \]
For this to be in $K$ requires that $b$ be integral, and so we may assume it is $0$.
Therefore if $g S(\gm) g^\vdash \in C$ we must have $g \in K$.  
Since $W_k$ is $K$-invariant, the quantity $W_k(g)$ is a constant multiple of $4k+1$.  As this depends on neither $g$ nor $\gm$, this factors out, and the result follows.

\end{proof}

\section{\bf Even Residue Characteristic}
\subsection{Computation of $\psi_k$}
We will make a convenient choice among the supercuspidal representations from Definition 1.6 of \cite{Ku}.
Recall that $\varpi$ is a uniformizer of $F$; let $E=F[\sqrt{\varpi}],$ and put $\varpi_E=\sqrt{\varpi}$.
We impose the condition that $2 \in \pp^2$.

Recall the subgroup $I_1$ of matrices in $K$ which are unipotent mod $\pp$.

Let $\Lambda_1$ be an additive character of $F$ whose kernel is $\pp$.

For $g \in I_1$, let $\lambda(g)=\Lambda_1(\tr(\varpi_E^{-1}(g-1))).$
This is a character of $I_1$, which is $I_2$-invariant.

First extend $\lambda$ to $G'=E^{\times}I_1$ via $\lambda(\gm g)=\lambda(g)$; this is again a well-defined character, and trivial on $E^{\times}$.
By Proposition 1.7 of \cite{Ku}, the representation $\pi_G=\cInd_{G'}^G \lambda$ is an irreducible supercuspidal representation.

[Translation:  To recover the notation of \cite{Ku} from ours, put $n=1,\rho=1, L=I_1,\alpha=\varpi_E,\Lambda_0(x)=\Lambda_1(\varpi^{-1}x), \chi=\Lambda_0,\lambda=1$.  
The following computations are pertinent in this regard:

\begin{itemize}
\item $E^{\times} \cap L=1+\pp_E$, 
\item $\tr(\varpi_E(x-1)) \in \pp^2$ if $x \in 1+\pp_E$.]
\end{itemize}

As before, we define $\psi$ to be the extension by $0$ of $\lambda$ on $G'$ to $G$, so that it is a matrix coefficient of $\pi_G$.  Please note that its central character is trivial.
Define $C_0=\OO_E^{\times} I_1$, $f_G$ and $f_0=\psi \cdot \xi_{C_0}$ in the same way as the previous section.  Note that $f_0$ is bi-$I_2$-invariant.
One checks that $\lambda^2=1$, so that $\pi_G$ is self dual.

\begin{prop} \label{a1} Suppose $\alpha \equiv 1 \Mod \pp$, with $\ord(\alpha-1)=e \geq 1$. 
Let $G_e^-= \{ g \in G| g=\kappa n_b a_e \text{, with } -e<\ord(b) \}$. 
Then
\[ \psi_k(\gm)= \int_{G_e^- T/T} W_k(g) \frac{dg}{dt}. \] \end{prop}

We will be more explicit below, but please note that this already implies that $\psi_k(\gm)>0$, since for $g \in G_e^-$, 
$w_k(g)= \ord(b)+e+2k+1>0$.

\begin{proof}

As in the proof of Proposition \ref{symmetric}, we may assume $g=\kappa n_b a_e$ with $ -e \leq \ord(b) \leq 0 $.
We have
\[ {}^p S(\gm)^{-1}=\left( \begin{array}{cc} 
\varpi^ex^{-1} &-b\varpi^e \\ 
0 & \varpi^e(-x^{-1}-1) \\
 \end{array} \right), \]
as before.  If $b \varpi^e \in \pp$ then ${}^p S(\gm)^{-1} \equiv \varpi^ex^{-1} \Mod I_2$.  A computation shows that
$\kappa \kappa^{\vdash} \equiv \det(\kappa) \Mod I_2$.  (This is where we use that $2 \in \pp^2$.)  Since $f$ is trivial on elements of $\OO^{\times},$ the proposition will follow from the following claim:

{\bf Claim: } $ \int_{G_eT/T} f(g S(\gm)^{-1} g^\vdash)W_k(g) \frac{dg}{dt}=0$,
where $G_e= \{g \in G| g=\kappa n_b a_e \text{, with } \ord(b)=-e \}$. 

Note that for $g \in G_e$, $W_k(g)=4k+1$ is a constant not depending on $g$, so we may replace it with $1$ in proving the claim.

We use the formula
\[ \int_{G_eT/T} f(g S(\gm)^{-1} g^\vdash)\frac{dg}{dt}=  \sum_{g \in K \bks G_eT /T} \int_{\kappa \in K}
f(\kappa g S(\gm)^{-1} (\kappa g)^\vdash) d\kappa, \]
and show that the inner integral is $0$ for all $g$.

Assuming $g \in G_e$, let $u=\varpi^ex^{-1}$ and $v=-b\varpi^e$; these are both units. 
Let $\kappa=\left( \begin{array}{cc} 
A & B \\ 
C & D \\
 \end{array} \right) \in K$.  Then a computation shows that $\Mod I_2$,
 
\[ \kappa S(\gm)^{-1} \kappa^{\vdash}=\left( \begin{array}{cc} 
A & B \\ 
C & D \\
 \end{array} \right)\left( \begin{array}{cc} 
u & v \\ 
0 & u \\
 \end{array} \right)\left( \begin{array}{cc} 
D & B \\ 
C & A \\
\end{array} \right)=\left( \begin{array}{cc} 
u(AD+BC)+ACv & A^2v \\ 
vC^2 & u(AD+BC)+ACv \\
 \end{array} \right). \]
This is in $I_1$ if and only if $C \in \pp$, which implies that $\kappa \in I_0$ and leads to the further simplification
\[ \left( \begin{array}{cc} 
ADu & A^2v \\ 
0 & ADu \\
 \end{array} \right)=ADu   \left( \begin{array}{cc} 
1 & \tfrac{Av}{Du} \\ 
0 & 1 \\
 \end{array} \right) . \]

Since $\lambda$ evaluated at this matrix is $\Lambda_1(\tfrac{Av}{Du})$, the inner integral over $K$ becomes

\[ \int_{I_0} \Lambda_1(\tfrac{Av}{Du}) d\kappa, \]
where $\kappa=\left( \begin{array}{cc} 
A & B \\ 
C & D \\
 \end{array} \right) \in I_0$.  
It is easy to see that $\kappa \mapsto \Lambda_1(\tfrac{Av}{Du})$ is a nontrivial homomorphism of the compact group $I_0$; therefore the integral over $I_0$ is $0$ for all $v$.  This finishes the proof of the claim.

\end{proof}

\subsection{Conclusion}

This section completes the computation of $\tilde{I}(s,f_G,1)$ in the case of characteristic $2$.
We begin with a review.

Let $F$ be a $\pp$-adic field of characteristic $2$, with $2 \in \pp^2$.  Let $\varpi$ be the uniformizer of $F$, and $E=F[\sqrt{\varpi}]$.
Let $G=\GL_2(F)$, and $H^+=O(w),$ the orthogonal group relative to the form $w=\left( \begin{array}{cc} 
0 & 1 \\ 
1 & 0 \\
 \end{array} \right)$.

Let $L'=M_n(\OO)$, and $L$ any open compact set of $M_n(F)$ containing $0$.
 
Let $\Lambda_1$ be an additive character of $F$ whose kernel is $\pp$.

For $g \in I_1$, let $\lambda(g)=\Lambda_1(\tr(\varpi_E^{-1}(g-1))).$
This is a character of $I_1$, which is $I_2$-invariant.

First extend $\lambda$ to $G'=E^{\times}I_1$ via $\lambda(\gm g)=\lambda(g)$. 

Define $\psi$ to be the extension by $0$ of $\lambda$ on $G'$ to $G$, and define $f=f_G$ as in Section \ref{f_G}.

\begin{thm} There are positive constants $A,B$ so that
 \[ \tilde{I}(s,f,1)=  |\OO^\times/\OO^{\times 2}| \sum_{k=0}^{\infty} q^{-4ks} (A+Bk). \]   \end{thm}

\begin{proof} Combining Propositions \ref{sort}, \ref{abig}, \ref{aother}, and \ref{a1}, we have

\[ \tilde{I}(s,f,1)=  \sum_{k=0}^{\infty} q^{-4ks} \int_{T_1} |D_\eps(\gm)|\int_{G_e^-T/T} W_k(g) \frac{dg}{dt} d\gm. \]
Here $T_1$ is the set of matrices in $T$ whose eigenvalues $\alpha, \alpha^{-1}$ are integral, and congruent to $1$ mod $\pp$.  We write $e=\ord(\alpha -1)$.  Then $G_e^-= \{ g \in G| g=\kappa n_b a_e \text{, with } -e<\ord(b) \}$.
In other words, if the Iwasawa decomposition of $g \in G$ can be written as $g=\kappa n_b a_e$ with $\kappa \in \GL_2(\OO)$,
$n_b=\left( \begin{array}{cc} 
1 & b \\ 
0 & 1 \\
 \end{array} \right)$, and $a_e=\left( \begin{array}{cc} 
\varpi^e & 0 \\ 
0 & 1 \\
 \end{array} \right)$, then $g \in G_e^-$ if and only if $-e<\ord(b)$.

For such $g$, we have $\Delta_1(g)=e+\ord(b)>0$, and $W_k(g)=|\OO^\times/\OO^{\times 2}|(2[e+\ord(b)]+4k+1)$.

Therefore, we may write
\[ \tilde{I}(s,f,1)=  |\OO^\times/\OO^{\times 2}|   \sum_{k=0}^{\infty} q^{-4ks} (A+Bk), \]
where
\[ A=\int_{T_1} |D_\eps(\gm)| \int_{G_e^-T/T} (2[e+\ord(b)]+1) \frac{dg}{dt} d\gm, \text{ and} \]
\[ B=4 \int_{T_1} |D_\eps(\gm)| \vol_{G/T} (G_e^-T/T) d\gm. \]

\end{proof}

\begin{comment}
\begin{defn} For integers $i \geq 0$, and $z \in \C$, let $G_i(z)= \sum_{k=0}^{\infty} k^i e^{-kz}$.  \end{defn}

\begin{lemma} $G_i(z)$ is the sum of $\frac{i!}{z^{i+1}}$ and a function which is holomorphic at $z=0$. \end{lemma}

\begin{proof} It is easy to see that $G_0(z)= \frac{1}{1-e^{-z}}$, which is the sum of $z^{-1}$ and a holomorphic function.  Also $G_{i+1}(z)=-G_i'(z)$.  The result follows. \end{proof}
\end{comment}

\bibliographystyle{plain}

\end{document}

\date{}

\title{Appendix:\ $L$-Functions and Poles of Intertwining Operators}

\author{By:\ F.~Shahidi\thanks{Partially supported by NSF grant DMS0700280}}

\maketitle
An important feature of Langlands--Shahidi method \cite{23} is that it connects
theory of automorphic $L$--function, both local and global, to the harmonic
analysis of the underlying group.
In particular, it relates the poles of local $L$--functions to those of
intertwining operators \cite{18,20}.

Following the lead of \cite{21,22}, Goldberg and Shahidi \cite{4,5,6} computed the
residues at the poles of intertwining operators in the cases of quasisplit classical groups, when the inducing
data is supercuspidal (not necessarily generic).
The aim of this note is to explain the number theoretic connections and how they
can predict results in harmonic analysis such as poles of intertwining operators and thus poles and zeros of Plancherel
measures, objects of significance in harmonic analysis and representation theory of $p$--adic groups, in a very precise
manner.

The work in \cite{25}, which this note is an appendix to, expresses this residue for certain maximal parabolic subgroups of quasisplit special orthogonal groups, as a ``weighted'' integral of corresponding matrix coefficients of inducing representations over a product of conjugacy and twisted conjugacy classes.
One hopes that this uniform expression is more amenable to interpretation of the residue in terms of endoscopy and $L$--functions.

To be able to address these connections, it is best to be more precise about the problems and the results.

Let $F$ be a $p$--adic local field of characteristic zero and let $\tilde G$ be a quasi--split classical group.
For simplicity of this exposition, let us leave out the case where $\tilde G$ is a unitary group (to split over a
quadratic extension).
Fix a Borel subgroup $B=TU$ of $\tilde G$, where $U$ its unipotent radical and $T$ a maximal torus there of.
Let $M$ be a Levi subgroup of a maximal parabolic subgroup $P=MN$ of $\tilde G$, $N\subset U$, uniquely fixed by
$M\supset T$.
Then $M=GL_n\times G$, where $G$ is a smaller classical group of the same type as $\tilde G$.
Let $\sigma\otimes \tau$ be a supercuspidal representation of $M(F)=GL_n (F)\times G(F)$ with $\sigma$ and $\tau$ one of
$GL_n (F)$ and $G(F)$, respectively.

Let $\rho$ be half the sum of roots of the maximal split subtorus $A_0$ of $T$ in $N$ and let $\alpha$ be the unique
simple root of $A_0$ among them.
Set $\tilde\alpha=\langle\rho,\alpha\rangle^{-1}\rho$.
If $X(M)_F$ is the group of $F$--rational characters of $M$, set
$$
\frak a=\text{ Hom }(X(M)_F,\bR)
$$
and its real dual
$$
\frak a^*=X (M)_F\otimes_{\bZ}\bR
$$
As usual, let
$$
H_M\colon M(F)\longrightarrow\frak a
$$
be defined by
$$
q^{\langle\chi, H_M (m)\rangle} = |\chi (m)|,
$$
for all $\chi\in X(M)_F$.
Here $q$ is the cardinality of the residue field of $F$ and $|\quad|$ is an absolute value on $F$ normalized by
$|\varpi|=q^{-1}$, where $\varpi$ is a prime element of $F$.

If $\langle\quad,\quad\rangle$ denotes the duality between $\frak a^*$ and $\frak a$, then
$q^{\langle\tilde\alpha,H_M(m)\rangle }$ and $q^{\langle s\tilde\alpha, H_M(m)\rangle }$ define a
real and a complex character of $M$, respectively, where $s\in\bC$ and thus $s\tilde\alpha\in
\frak a^*_{\bC}=\frak a^*\otimes_{\bR}\bC$.
Let
\begin{equation}
I(s,\sigma\otimes\tau)=\text{ Ind}^{G(F)}_{P(F)}\sigma\otimes\tau\otimes q^{\langle s\tilde\alpha, H_M(\ )\rangle}
\end{equation}
and
\begin{equation}
I(\sigma\otimes\tau)=I(0,\sigma\otimes\tau).
\end{equation}
Finally, let
\begin{equation}
A(s,\sigma\otimes\tau,w_0) f(g)=\int_{N'(F)} f(w_0^{-1} n' g) dn',
\end{equation}
$f\in V(s,\sigma\otimes\tau)$, the space of $I(s,\sigma\otimes\tau)$, 
where $w_0$ is a representative for the long element of $W(A_0,G)$ modulo that of $W(A_0,M)$.
Moreover $N'=w_0 N^- w_0^{-1}$, where $N^-$ is the opposite of $N$.
When $P$ is self--associate, which is the case unless $\tilde G$ is a $SO_{2k}$ with $k$ odd and $M=GL_k$, $N'=N$.
Then (3)converges absolutely for $Re(s)>0$ and extends to a rational function of $q^{-s}$ with possible poles only at $s$ with
$q^s=1$.
One can normalize to study the pole at $s=0$ by appropriately twisting $\sigma\times\tau$ by an unramified character $q^{\langle
s_0\tilde\alpha,H_M(\ )\rangle}$ for some $s_0\in\bC$ to assume $s=0$.

It is an important consequence of Harish--Chandra's work on $p$--adic groups that:\ $I(\sigma\otimes\tau)$ {\it is reducible if and only if
$\sigma\cong\tilde\sigma$ and $A(s,\sigma\otimes\tau,w_0)$ is holomorphic at} $s=0$ (cf.~\cite{7,24}).

The goal of \cite{4,5,6,21,22} was then to compute the residue at this pole, not only for determining the reducibility points, but to
determine the local $L$--functions defined through the Langlands--Shahidi method.
In fact, assume $\tau$ is generic and let $L(s,\sigma\times\tau)$ be the Rankin product $L$--function attached to the pair $(\sigma,\tau)$
in \cite{20} for each of the cases with a $B, C$ or $D$ Dynkin diagram in the tables of \cite{19}.
Moreover, denote by $L(s,\sigma,\Lambda^2)$ and $L(s,\sigma,\Sym^2)$ the exterior and the symmetric square $L$--functions attached to
$\sigma$ in \cite{21}.
In fact, if
\begin{equation}
\rho_\sigma\colon W_F\longrightarrow GL_n (\bC)
\end{equation}
is the $n$--dimensional complex representation of $W_F$, parametrizing $\sigma$, \cite{8,9} then as it is verified in \cite{10}
\begin{equation}
L(s,\sigma,\Lambda^2)=L(s,\Lambda^2\cdot\rho)
\end{equation}
and
\begin{equation}
L(s,\sigma,\Sym^2)=L(s,\Sym^2\cdot\rho),
\end{equation}
where the $L$--functions on the right are those of Artin.
Then one important consequence of the theory developed in \cite{18,19,20} is that
\begin{equation}
L(s,\sigma\times\tilde\tau) L(2s,\sigma,r) A(s,\sigma\otimes\tau,w_0)
\end{equation}
is a holomorphic and non--zero operator on all of $\bC$.
Here $r=\Lambda^2$ unless $\tilde G$ is an odd special orthogonal group in which case $r=\Sym^2$.
We observe that $r=\Lambda^2$ if and only if $^LG^0$, the connected component of the $L$--group of $G$, is a special orthogonal group.

There are many instances where the $L$--functions appearing in (7) are already well--understood through other means, e.g., global theory \cite{2}.
Equation (7) then immediately determines both the poles of $A(s,\sigma\otimes\tau,w_0)$ and the reducibility of
$I(\sigma\otimes\tau)$.
This then allows us to predict the vanishing or the non--vanishing of the residues determined in \cite{4,5,6}.

As it is evident from the residue calculations in \cite{4,5,6}, the best way to relate these residues with what $L$--functions predict is
to introduce the two regular and singular terms.
More precisely, let $f'\in C_c^\infty (GL_n (F))$ define a matrix coefficient of $\sigma$ by means of integration along the center
$F^*$ of $GL_n(F)$.
Also denote by $f_\tau\in C_c^\infty (G(F))$ a matrix coefficient of $\tau$.
As explained in \cite{4,5,6}, one can then write the basic residues of (3) at $s=0$, as a sum of a regular term
\begin{equation}
(2n\log q)^{-1} R_G (f_\tau,f')
\end{equation}
and a singular one
\begin{equation}
R_{\sin \! g} (f_\tau,f').
\end{equation}

One then expects that if there are poles for some datum, then either (8) or (9) would be non--vanishing and only one of them.
In fact, by its definition the non--vanishing of (8) which looks like a Weyl integration formula, giving a pairing between orbital
integrals of $f_\tau$ and the twisted orbital integrals of $f'$ through an integral over $F$--points of elliptic Cartan subgroups of
$G$, and with the conjugacy and the twisted conjugacy classes in $G(F)$ and $GL_n(F)$ corresponding to each other through an image (or norm) map \cite{11}, should be equivalent to a pole for $L(s,\sigma\times\tilde\tau)$.
The pole of $L(2s,\sigma,r)$ at $s=0$ should account for the non--vanishing of (9).

This is clearly conjectural and has been hard to verify.
The purpose of Spallone's work in \cite{25} is to make these residues amenable to verifying these conjectural connections.
In particular, one hopes that his Theorem 1 can be used to verify the case of
$M=GL_2\times SO_2^*$ as a Levi subgroup of a maximal parabolic inside $\tilde
G=SO_6^*$, where $SO^*$ denotes the quasisplit special orthogonal group defined by a fixed quadratic extension $E$ of $F$, by means of explicit calculations of matrix coefficients which are available in these low dimensional cases.
We shall now proceed to explain the results deduced using $L$--functions in this case.

Let $\omega=\omega_\sigma$ be the central character of $\sigma$.
Then $L(s,\sigma,r)=L(s,\sigma,\Lambda^2)=L(s,\omega)$, the local Hecke--Tate $L$--function attached to $\omega$.
Moreover $\tilde\sigma\cong\sigma$ is simply equivalent to $\omega^2=1$.

In fact, if $\{e_i^\vee\pm e_j^\vee|1\leq i < j\leq 3\}$ denotes the positive roots for $\hat{\tilde G} =^L\!\tilde G^0$, the connected 
component of the $L$--group of $\tilde G$, then
\begin{equation}
^L\frak n=\bC X_{e_1^\vee+e_2^\vee} +\bC X_{e_1^\vee\pm e_3^\vee}+\bC X_{e_2^\vee\pm e_3^\vee},
\end{equation}
where each $X_*$ denotes a fixed root vector in the corresponding 1--dimensional root space of $^L\! \frak n$, is five dimensional.
We assume that the set of these root vectors is Gal$(E/F)$--invariant.
Moreover $\hat M=GL_2(\bC)\times (GL_1(\bC)\times GL_1 (\bC))/\bC^*$, while $^LM=\hat M\times W_F$ in which $W_F$ acts through Gal$(E/F)$ trivially, if $w\in W_F$ goes to the trivial element of Gal$(E/F)$, and it permutes the two copies of $GL_1(\bC)$, otherwise.
Note that the non--trivial element of Gal$(E/F)$ will fix $e_1^\vee$ and $e_2^\vee$, but sends $e_3^\vee$ to $-e_3^\vee$.
In particular, $W_F$ fixes $X_{e_1^\vee+e_2^\vee}$ and sends $X_{e_i^\vee\pm e_3^\vee}$ to $X_{e_i^\vee\mp e_3^\vee},i=1,2$, thus leaves $\langle X_{e_1^\vee+e_2^\vee}\rangle$ and $\langle X_{e_1^\vee\pm e_3^\vee},X_{e_2^\vee\pm e_3^\vee}\rangle$ invariant.
The representation $r$ of $^L\!M$ on $^L\!\frak n$ will reduce to $r_2$ on $\langle X_{\alpha_1^\vee+\alpha_2^\vee}\rangle$ through the projection of $^L\!M\to GL_2(\bC)$ composed with $\Lambda^2=\det\colon GL_2(\bC)\to GL_1(\bC)$ and $r_1$ on the remaining four dimensional complex subspace $\langle X_{e_1^\vee\pm e_3^\vee}, X_{e_2^\vee\pm e_3^\vee }\rangle$ of $^L\!\frak n$.
While $W_F$ acts reducibly on the latter space, the action $r_1$ of $^L\!M$ on it is irreducible.
To show that $r_1$ is the required tensor product, we need to determine $\hat r_1=r_1|\hat M$.
We first identify $(GL_1(\bC)\times GL_1(\bC))/\bC^*$ with $GL_1(\bC)$ by sending $(a,b)_*$ to $a/b$, where $(a,b)\in GL_1(\bC)\times GL_1(\bC)$ is sent to $(a,b)_*\in (GL_1 (\bC)\times GL_1 (\bC))/\bC^*$.

If we realize $SO^*_2$ defined by $E/F$ with an isotropic Cartan subgroup of $GL_2$, then the identification ``$(a,b)_*\mapsto a/b$'' is $\alpha^\vee$, the root character of $PGL_2(\bC)$.
In particular, if $\chi$ is a character of $SO_2^*(F)=E^1$, the group of elements of norm 1 in $E^*$, then the corresponding homomorphism \cite{14,15} $\chi^\vee\colon W_F\to (GL_1(\bC)\times GL_1(\bC))/C^*$ satisfies
\begin{equation}
\alpha^\vee (\chi^\vee (w))=\chi(\alpha^\vee (w)),
\end{equation}
$w\in E^*$, where $\alpha^\vee$ on the right is the coroot map into $SO_2^*$.
Thus $\chi(\alpha^\vee(w))=a_w/b_w$, $w\in E^*$.

Next, we identify $\hat m=(g,(a,b)_*)$ with 
$$
\text{diag}(g,ab^{-1},ba^{-1},\ w^t g^{-1} w^{-1})\in SO_6(\bC),
$$
where $w=\begin{pmatrix}0&1\\ 1&0\end{pmatrix}$.
Then
$$
\hat r_1 (\hat m)\cdot X=g X(a/b)^{-1}
$$
for
$$
X\in\langle X_{e_1^\vee-e_3^\vee},\ X_{e_2^\vee-e_3^\vee}\rangle
$$
and
$$
\hat r_1 (\hat m)\cdot X=gX (b/a)^{-1}
$$
for
$$
X\in\langle X_{e_1^\vee+e_3^\vee}, X_{e_2^\vee+e_3^\vee}\rangle.
$$
Let $\rho_2$ be the standard representation of $GL_2(\bC)$.
We have therefore concluded that
\begin{equation}
\hat r_1=\rho_2\otimes\alpha^\vee\oplus\rho_2\otimes(\alpha^\vee)^{-1}.
\end{equation}
Taking into account the action of $W_F$, we then have
\begin{equation}
r_1=\rho_2\otimes\text{ Ind}_{SO_2(\bC)}^{SO_2(\bC)\rtimes\text{ Gal}(E/F)}\alpha^\vee,
\end{equation}
$SO_2(\bC)=(GL_1(\bC)\times GL_1 (\bC))/C^*$, using the finite Galois form of $^L\!SO_2^*$ and realizing $\alpha^\vee$ as a character of $SO_2(\bC)$.

It now follows from (11) and (13) that
\begin{equation}
L(s,\sigma\otimes\tilde\tau,r_1)=L(s,\sigma\times\chi^{-1})
\end{equation}
where $\tau=\chi$ is a character of $E^1=SO_2^*(F)$, i.e., the Rankin product $L$--function for $\sigma$ and $\chi^{-1}$.
In fact, one can check that the composite of $\chi^\vee\colon W_F\to ^L\!SO_2^*$ with the induced representation in (3) equals
\begin{equation}
\rho_{E/F}(\chi)=\text{ Ind}_{E^*}^{W_{E/F}}\alpha^\vee\cdot\chi^\vee=\text{ Ind}_{E^*}^{W_{E/F}}\chi\cdot\alpha^\vee
\end{equation}

The coroot map $\alpha^\vee$ can be given more explicitly.
It is, in fact, the map $\alpha^\vee\colon x\mapsto x/x^\nu$, $1\neq\nu\in\text{ Gal}(E/F)$, which by the Hilbert's Theorem 90 is a surjection from $E^*$ onto $E^1$.
Its kernel is $F^*$ and therefore $\alpha^\vee$ induces an isomorphism $E^*/F^*\cong E^1$.
Thus there is a one--one correspondence between characters $\chi$ of $SO_2^*(F)=E^1$ and those of $E^*$ trivial on $F^*$, $\chi\leftrightarrow\tilde\chi$, namely $\tilde\chi=\chi\cdot\alpha^\vee$.
We finally note that the Weil representation of $GL_2(F)$, attached to the representation $\Ind^{W_{E/F}}_{E^*}\tilde\chi$ of $W_{E/F}$ or $W_F$ through its natural projection, has $\eta=\eta_{E/F}$, the quadratic character of $F^*$ attached to $E/F$ by class field theory, as its central character, since $\tilde\chi|_{F^*}=1$.
(The central character is $\eta_{E/F}\cdot\tilde\chi|_{F^*}$ in general.)

Let $I_{E/F}(\tilde\chi^{-1})$ denote the automorphic induction of $\tilde\chi^{-1}$ as a representation of $GL_2(F)$, i.e., the one parametrized by $\Ind_{E^*}^{W_{E/F}}\tilde\chi^{-1}$.
Using a standard local--global argument (cf.~\cite{20}) for $L(s,\sigma\otimes\chi^{-1},r_1)$ and $L(s,\sigma\times I_{E/F}(\tilde\chi^{-1}))$, the Rankin product $L$--function for the pair $\sigma$ and $I_{E/F}(\tilde\chi^{-1})$ on $GL_2(F)\times GL_2(F)$, one can easily show

\medskip\noindent
{\bf Proposition 1}.
$$L(s,\sigma\otimes\chi^{-1},r_1)=L(s,\sigma\times I_{E/F}(\tilde\chi^{-1})).$$

We observe that by \cite{12,16} the $L$--functions in Proposition 1 are equal to the Artin $L$--functions $L(s,\rho_\sigma\otimes\rho_{E/F}(\chi^{-1}))$, where $\rho_\sigma$ is the 2--dimensional complex representation of $W_F$ attached to $\sigma$ as in (4).
Moreover, in the proof of Proposition 1, one needs to use the fact that the automorphic induction for a gross\"encharacter that has $\tilde\chi$ as local component always exists.

Using base change and automorphic induction \cite{17}, it is clear that 
\begin{equation}
L(s,\sigma\otimes\chi^{-1},r_1)=L_E (s,BC_{E/F}(\sigma)\otimes\tilde\chi^{-1}),
\end{equation}
where $BC_{E/F}(\sigma)$ is the base change of $\sigma$ as a representation of $GL_2(E)$ (cf.~\cite{1,17}).

For the last $L$--function to have a pole at $s=0$, it is clear that $BC_{E/F}(\sigma)$ must become a principal series representation of $GL_2(E)$ induced from a pair of character of $E^*$, with $\tilde\chi$ one of them.
By Theorem 4.2 of Chapter 3 of \cite{1} or \cite{17} $\sigma$ must be dihedral with respect to $\eta=\eta_{E/F}$, i.e., $\sigma\cong\sigma\otimes\eta$.
Moreover $BC_{E/F}(\sigma)=\Ind(\gamma,\gamma^\nu)$, where $\gamma$ is a character of $E^*$.
Thus $BC_{E/F}(\sigma)=\Ind(\tilde\chi,\tilde\chi^\nu)$.

If the residual characteristic $p\neq 2$, then every supercuspidal representation of $GL_2(F)$ will be dihedral.
On the other hand, when $p=2$, there are non--dihedral supercuspidal representations which are called exceptional or extraordinary \cite{13}.

A dihedral $\sigma$ is always attached through the local Langlands conjecture to an irreducible representation of $W_F$ of the form
$$
I(W_F,W_K,\theta)=\Ind_{K^*}^{W_{K/F}}\theta, 
$$
where $K/F$ is determined by the non--trivial quadratic character $\eta$ and $\theta$ is a character of $K^*$ which does not factor through $N_{K/F}$.
It is therefore a Weil representation in the classical sense.
We recall again that the central character $\omega_\sigma=\eta\cdot\theta|_{F^*}$.

Suppose $\sigma\cong\tilde\sigma$ which implies $\omega_\sigma^2=1$.
If $\omega=\omega_\sigma=1$, then $\sigma$ may be considered as a representation of $PGL_2(F)=SO_3(F)$, noting that $SO_3$ is a twisted endoscopic group (cf.~\cite{11}) of $GL_2$ with respect to automorphism 
\begin{equation}
\theta(g)=\begin{pmatrix}0&1\\ -1&0\end{pmatrix} \!{}^tg^{-1}
\begin{pmatrix}0&1\\ -1&0\end{pmatrix}^{-1}.
\end{equation}
In this case $L(s,\sigma,\Lambda^2)=L(s,\omega_\sigma)$ will have a pole at $s=0$.

Using the holomorphy and non--vanishing of the operator (7) on all of $\bC$, we can now put Proposition 1 and all our calculations and discussions together to conclude:

\medskip\noindent
{\bf Proposition 2}.
{\it Let $\sigma$ be an irreducible unitary supercuspidal representation of $GL_2(F)$ and $\chi$ a character of $E^1$, the subgroup of elements of norm one in $E^*$, where $[E\colon F]=2$.
Consider $\sigma\otimes\chi$ as a representation of $GL_2(F)\times SO_2^* (F)$ realized as the $F$--points of a Levi subgroup of a maximal parabolic subgroup of $SO_6^*$, the quasisplit form of $SO_6$ defined by $E/F$.
Then the standard intertwining operator $A(s,\sigma\otimes\chi,w_0)$ is holomorphic at $s=0$, unless $\sigma\cong\tilde\sigma$ and

a) $\omega_\sigma=1$, i.e., $\sigma$ comes from $SO_3$, or

b) $E/F$ is the extension determined by $\omega_\sigma$ and $BC_{E/F}(\sigma)=\Ind(\tilde\chi,\tilde\chi^\nu)$, where Gal$(E/F)=\{1,\nu\}$ and $\tilde\chi=\chi\cdot\alpha^\vee$, i.e., $\sigma$ comes from $SO_2^*$.
Moreover, $I_{E/F}(\tilde\chi)$ is the twisted endoscopic transfer of $\chi$ from $SO_2^*(F)$ to $GL_2(F)$ for each $\chi$.

In both cases $A(s,\sigma\otimes\chi,w_0)$ has a simple pole at $s=0$ and only then.}

\medskip\noindent
{\bf Remark 3}.
{\it We note that if $\omega_\sigma\neq 1$, then the pole of $L(s,\sigma\times\tilde\sigma)$, the Rankin product $L$--function attached to $\sigma$ and $\tilde\sigma$, at $s=0$, will force $L(s,\sigma,\Sym^2)$ to have a pole there.
In particular, $\sigma$ would be dihedral \cite{17}.
The significance of Proposition 2 is that for $A(s,\sigma\otimes\chi,w_0)$ to have a pole, the inducing data for dihedral $\sigma$ must be $\tilde\chi=\chi\cdot\alpha^\vee$.
One can therefore reformulate part b) of the proposition as:}

\medskip\noindent
{\bf Proposition 4}.
{\it Condition b) of Proposition 2 is equivalent to:

b$'$) $E/F$ is the extension determined by $\omega_\sigma$ and the dihedral representation $\sigma$ is attached to
$\Ind_{W_E}^{W_F}\tilde\chi$, $\tilde\chi=\chi\cdot\alpha^\vee$, through the local Langlands conjecture.}

\medskip\noindent
{\bf Proposition 5}.
{\it The induced representation $I(\sigma\otimes\chi)$ is irreducible unless $\sigma\cong\tilde\sigma$.
Suppose $\sigma\cong\tilde\sigma$.
Then $I(\sigma\otimes\chi)$ is irreducible if and only if one of the conditions a), b) or b$'$) of Propositions 2 and 4 holds.}

\medskip\noindent
{\bf Remark 6}.
{\it The character $\tilde\chi$ factors through the norm if and only if $\tilde\chi^\nu=\tilde\chi$ or $\chi(y)=\chi(y^{-1})$ for all $y\in E^1$ or $\chi^2=1$.
Thus $I_{E/F}(\tilde\chi)$ is supercuspidal if and only if $\chi^2\neq 1$.
Note that they are all self dual since their common central character $\eta_{E/F}$ is quadratic.
$(\tilde\chi=\tilde\chi^{-\nu}$ is automatically valid.)}

\medskip\noindent
{\bf Remark 7}.
{\it The general formulation of these results for $SO^*_{2n}$ of arbitrary rank requires the knowledge of generic transfer from $SO^*_{2n}(F)$ to $GL_{2n}(F)$.
This is the subject matter of a work in progress of author with J.~Cogdell, which as an application will include these generalizations.
As observed here, when $n=1$, this transfer is no more than an automorphic induction for the character $\tilde\chi=\chi\cdot\alpha^\vee$, where $\chi$ is one of $SO_2^*(F)$.}

\medskip\noindent
{\bf Remark 8}.
{\it Using twisted characters \cite{1,17} for the representation $\Sigma$ of the disconnected group $GL_2(F)\times\langle\theta\rangle$ extending $\sigma$ which is possible since $\tilde\sigma\cong\sigma$, where $\theta$ is as in (17), one expects a Shintani type character identity as in Lemma 11.2 of (17) or Theorem 6.2 of \cite{1} to hold for the case of twisted endoscopy, if $\omega_\sigma^2=1$, but $\omega_\sigma\neq 1$.
A careful analysis of this character identity must then allow us to show $R_G(\overline\chi,f')\neq 0$ for a choice of data if and only if $\sigma$ comes from $\chi$ as in Part b) of Proposition 2.

If Theorem 1 of \cite{25} can be used to prove that if $\sigma$ has a trivial central character, i.e., it comes from $SO_3(F)\cong PGL_2(F)$, then the full residue, i.e., the sum of (8) and (9) (also cf.~Theorem 1 of \cite{25}) is non--zero, then in view of the vanishing of (8) by means of Shintani type identities, one immediately concludes that (9) is non--zero.

It may be possible to conclude that the non--vanishing of (8) implies the vanishing of (9) for function theoretic reasons, but that remains to be seen.}

\numberwithin{equation}{section}
\addtocounter{section}{9}
\setcounter{equation}{0}
\medskip\noindent
{\bf Remark 9}.
{\it Here we like to make some comments concerning the use of our results in \cite{4} by Spallone in \cite{25}.
It mainly concerns the fact that compact subgroups of unipotent radicals of classical groups are more complicated than those of $GL(N)$.
The correct way is to use the Lie algebra and the exponential map of the unipotent radical to define these subgroups.
This was not emphasized in \cite{4} and we will therefore discuss this here and show how they justify the results used from
\cite{4} in \cite{25}.

With notation as in \cite{4}, suppose $n=2m$.
We want to explain how the compact sets $L$ and $L'$ introduced in page 263 of \cite{4}, from which only $L'$ is open in $M_n(F)$, can be constructed using Lie algebras.

Let $\frak n(F)=$ Lie $(N(F))$.
Then $\frak n(F)\subset M_n(F)^2$.
For an integer $r$, let $\frak n(P^r)=\frak n(F)\cap M_n(P^r)^2$.
If $n(X,Y)\in N(F)$, then the exponential map sends $(X,\tilde Y)\in\frak n(F)$ to $(X,Y)\in N(F)$, where $(X,\tilde  Y)$ and $(X,Y)\in M_n (F)^2$ with
\begin{equation}
Y+\tilde\var (Y)=XX'\label{9.1}
\end{equation}
(cf.~equation (2.1) of \cite{4}), while $(X,\tilde Y)$ runs over all of $M_n(F)^2$ only subject to
\begin{equation}
\tilde\var (\tilde Y)=-\tilde Y
\end{equation}
with no restriction on $X$.
The exponential map exp$\colon\frak n(F)\longrightarrow N(F)$ simply satisfies
\begin{equation}
Y=\tilde Y+{1\over 2}\ XX'.
\end{equation}

Next let
\begin{equation}
\tilde L_r=\{\tilde Y\in M_n (P^r)|\tilde\var (\tilde Y)=-\tilde Y\}
\end{equation}
and let $\tilde\xi_r$ be its characteristic function.
Set $\tilde L'_r=M_n(P^r)$ and denote by $\tilde\xi'_r$ its characteristic function.
If
\begin{equation}
\frak n(F)\ni (X,\tilde Y)\overset{{\rm exp}}{\longrightarrow} (X,Y)\in N(F),
\end{equation}
then the functions $\xi_L$ and $\xi_{L'}$ defined in Lemma 2.3 of \cite{4} (Page 263) are in fact defined by
\begin{eqnarray}
\xi_L (Y)\xi_{L'} (X)\colon=\xi_r (n(X,Y))\ \ \\
\colon=(\tilde\xi'_r,\tilde\xi_r) (X,\tilde Y)\nonumber\\
\colon=\tilde\xi'_r (X)\tilde\xi_r(\tilde Y).\ \ \nonumber
\end{eqnarray}
for some $r\in\bZ$.
Observe that while $X$ and $\tilde Y$ can change independent of each other and only subject to equation (9.2), $X$ and $Y$ depend on each other through (9.1).
We therefore must consider $\xi_{L'}(X)\xi_L(Y)$ defined by means of equation (9.6) as
\begin{equation}
\xi_{L'} (X) \xi_L (Y):\ =\tilde{\xi'}_r (X)\tilde\xi_r (\tilde Y).
\end{equation}

To bring this discussion to bear on equation (4.5) of \cite{4} in Page 283, which is used in \cite{25}, we must see how
\begin{equation}
\xi_{L'} (z^{-1} g h_0 h)\xi_{\tilde\var(L)} (z^{-2} g Y\var (g)^{-1})
\end{equation}
depends on $|z|$.
Note that $\tilde\var (\tilde L_r)=\tilde L_r$.
Using definitions (9.5), (9.6) and (9.7) we have:
\begin{equation}
\xi_{L'} (z^{-1} g h_0 h)\xi_{\tilde\var(L)} (z^{-2} g Y\var (g)^{-1})
\end{equation}
equals
\begin{eqnarray}
&&(\tilde\xi'_r,\tilde\xi_r)(\log (n(z^{-1} g h_0 h, z^{-2} g Y\var (g)^{-1})))\\
&=&\tilde\xi'_r (z^{-1} g h_0 h)\tilde\xi_r (z^{-2} g\tilde Y\var (g)^{-1}),\nonumber
\end{eqnarray}
using the easily checked identity
\begin{equation}
\widetilde{gY\var (g)^{-1}}=g\tilde Y\var (g)^{-1},
\end{equation}
and by absorbing $z^{-2}$ into $g$ and $\var(g)^{-1}$ by $z^{-1}g$ and $\var(z^{-1}g)^{-1}$, respectively.

If $gY\var (g)^{-1}$ lies in a compact set in $GL_n(F)$, then $g\tilde Y\var (g)^{-1}\in\tilde L_t$ for some $t\in\bZ$.
Thus for $|z|$ appropriately large, depending on the support of $f'$ (with the notation as in \cite{4} or $f_G$ with that of \cite{25}), but not on $Y$ or $g$, $\tilde\xi_r (z^{-2} g\tilde Y\var (g)^{-1})=1$ and thus the last two integrals in equation (4.5) of \cite{4} reduce to 
\begin{equation}
\int_T \int_{Z_k}\xi_{L'} (z^{-1} g th)|\det g|^{-2s} dt dz
\end{equation}
for $h_0=t\in T$ and $z\in Z_k$, where
\begin{equation}
Z_k=\{z\mid |z|=q^k\},
\end{equation}
$k\in\bZ$, all in the notation of \cite{25}.
Thus
\begin{equation}
\tilde w_k (g,h)=w_k (g,h)
\end{equation}
again as claimed in \cite{25}.}

\bigskip\noindent
Dept.~of Mathematics\newline
Purdue University\newline
West Lafayette, IN\ 47907-1395\newline
shahidi@math.purdue.edu